%% file: main.tex
\documentclass[11pt]{amsart}
\usepackage[utf8]{inputenc}
\usepackage{graphicx} 

\usepackage{fullpage, amsmath, amssymb, latexsym, amsfonts, amscd, amsthm, extarrows, mathtools, tikz-cd}

\usepackage{upgreek}
\usepackage{comment}
\usepackage{enumitem}

\usepackage{hyperref}
\hypersetup{pdftoolbar=true, pdftitle={Template}, pdffitwindow=true, colorlinks=true, citecolor=blue, filecolor=black, linkcolor=black, urlcolor=red, hypertexnames=false}

\usepackage{xcolor}

\newtheoremstyle{thmlike}
{8pt}
{3pt}
{\slshape}
{}
{\bfseries}
{.}
{1em}
{}

\newtheoremstyle{deflike}
{8pt}
{3pt}
{}
{}
{\bfseries}
{.}
{1em}
{}

\theoremstyle{thmlike}
\newtheorem*{theorem*}{Theorem}
\newtheorem{theorem}{Theorem}[section]
\newtheorem{proposition}[theorem]{Proposition}
\newtheorem{lemma}[theorem]{Lemma}
\newtheorem{corollary}[theorem]{Corollary}

\theoremstyle{deflike}
\newtheorem{definition}[theorem]{Definition}
\newtheorem{remark}[theorem]{Remark}

\newtheorem{hypothesis}[theorem]{HYPOTHESIS}

\newcommand{\opp}{\mathrm{op}}
\newcommand{\Pord}{P-\mathrm{ord}}

\newcommand{\sgl}{\mathrm{Sgl}}

\newcommand{\tp}[1]{\prescript{t}{}{#1}} 

\newcommand{\diag}{\operatorname{diag}}

\newcommand{\End}{\operatorname{End}}

\renewcommand{\ker}{\operatorname{ker}}
\renewcommand{\hom}{\operatorname{Hom}}

\newcommand{\isom}{\operatorname{Isom}}

\newcommand{\supp}{\operatorname{supp}}

\newcommand{\Gm}{\operatorname{\mathbb{G}_{m}}}

\newcommand{\Gal}{\operatorname{Gal}}
\newcommand{\GL}{\operatorname{GL}}

\newcommand{\Lie}{\operatorname{Lie}}

\newcommand{\Char}{\operatorname{char}}
\newcommand{\cond}{\operatorname{cond}}
\renewcommand{\det}{\operatorname{det}}
\newcommand{\nm}{\operatorname{Nm}}
\newcommand{\ord}{\operatorname{ord}}
\newcommand{\rk}{\operatorname{rank}}
\newcommand{\tr}{\operatorname{tr}}
\renewcommand{\v}{\operatorname{\nu}} 
\newcommand{\vol}{\operatorname{Vol}}

\newcommand{\ind}[2]{\operatorname{\iota}_{#1}^{#2}}
\newcommand{\Ind}{\operatorname{Ind}}
\newcommand{\Rep}{\operatorname{Rep}}

\newcommand{\cg}[1]{\widetilde{#1}} 

\renewcommand{\AA}{\mathbb{A}}

\newcommand{\CC}{\mathbb{C}}

\newcommand{\GG}{\mathfrak{G}}

\newcommand{\JJ}{\mathcal{J}}
\newcommand{\KK}{\mathcal{K}}

\newcommand{\OO}{\mathcal{O}}
\newcommand{\PP}{\mathcal{P}}
\newcommand{\QQ}{\mathbb{Q}}
\newcommand{\bQQ}{\overline{\QQ}}
\newcommand{\RR}{\mathbb{R}}

\newcommand{\ZZ}{\mathbb{Z}}
\newcommand{\bZZ}{\overline{\ZZ}}

\renewcommand{\i}{\iota}

\newcommand{\g}{\mathfrak{g}} 

\newcommand{\p}{\mathfrak{p}} 
\newcommand{\s}{\mathfrak{s}}
\newcommand{\X}{\mathfrak{X}}

\renewcommand{\d}{\bold{d}} 

\newcommand{\la}{\langle} 
\newcommand{\ra}{\rangle}

\newcommand{\incl}{\mathrm{incl}}
\newcommand{\pr}{\mathrm{pr}} 

\newcommand{\brkt}[2]{\la #1, #2 \ra}
\newcommand{\absv}[1]{\left|#1\right|}

\newcommand {\ol}[1] {\overline{#1}}
\newcommand {\ul}[1] {\underline{#1}}

\newcommand{\at}{\makeatletter @\makeatother}

\usepackage[textwidth=14cm]{geometry}
\usepackage{fancyhdr}
\begin{document}
\pagestyle{fancy}
\fancyhf{}
\fancyhead[RO,LE]{\footnotesize\thepage}
\fancyhead[CE]{\footnotesize\leftmark}
\fancyhead[CO]{\footnotesize $p$-ADIC ZETA INTEGRALS}
\renewcommand{\headrulewidth}{0pt}

\title{$p$-adic zeta integrals on unitary groups via Bushnell-Kutzko types}
\author[D. Marcil]{David Marcil}
\date{\today}
\address{David Marcil, Department of Mathematics, Columbia University, New York, NY 10027, USA}
\email{d.marcil\at columbia.edu}
\subjclass[2010]{Primary: 11F85, 11S40; Secondary: 11F55, 11F66.}
\keywords{Bushnell-Kutzko types, $P$-ordinary representations, zeta integrals, $L$-functions.}

\maketitle

\begin{abstract}
    In this paper, we compute certain $p$-adic zeta integrals appearing in the doubling method of Garrett and Piatetski-Shapiro–Rallis for unitary group. Using structure theorems in \cite{Mar23a} for $P$-(anti-)ordinary automorphic representations involving Bushnell-Kutzko types, we associate local Siegel-Weil sections at $p$ to such Bushnell-Kutzko types. Then, fixing compatible choices of $P$-anti-ordinary vectors, we find explicit formulae relating the corresponding $p$-adic zeta integral to modified $p$-Euler factors and volumes of $P$-Iwahoric subgroups. Our results extend the ones of \cite[Section 4.3]{EHLS} by allowing automorphic representations with nontrivial supercuspidal support at $p$.
\end{abstract}

\tableofcontents

\section*{Introduction}
In recent years, the so-called \emph{doubling method} of Garrett and Piatetski-Shapiro--Rallis \cite{Gar84, GPSR87} has been used in several situations to compute special values of $L$-functions and construct $p$-adic $L$-functions, see \cite{Liu19, Liu21, LiuRos20, EisLiu20, EHLS}.

One of the major accomplishment of \cite{EHLS} is the explicit computation of the local $p$-adic zeta integral provided by the doubling method when working with ordinary automorphic representation on a unitary group $G$ over a CM field $\KK$. They use this result, together with many other technical considerations, to construct a $p$-adic $L$-functions of Hida families of such representations. The goal of this paper is to extend the results of \cite[Section 4.3]{EHLS} by considering $P$-ordinary representations instead, where $P$ is some parabolic subgroup of $G_{/\ZZ_p}$. We recover the formulae of \emph{loc. cit.} when $P$ is essentially a product $B$ of upper triangular Borel subgroups. In other words, ordinary representations are equivalently $B$-ordinary and our results agree with known formulae when $P=B$.

This completes crucial computations in an ongoing project of the author to construct $p$-adic $L$-functions for $P$-ordinary families on $G$. If the center of a Levi factor of $P$ has rank $d$, then these are $(d+1)$-variable $p$-adic $L$-functions. The underlying goal is again to extend the main result of \cite{EHLS}. However, the author also hopes to use general techniques developed here to deal with $p$-adic integrals in other projects.

\subsubsection*{Main result.} 
In general, if $\pi$ is an ordinary cuspidal automorphic representation of $G$, then its local $p$-factor $\pi_p$ has trivial supercuspidal support and this leads to the existence of an \emph{ordinary nebentypus character} $\mu$ associated to $\pi$. This character $\mu$, as well as existence of ordinary vectors of $\pi_p$ on which Iwahori subgroups act via $\mu$, play key roles in the computations of \cite{EHLS}. However, if $\pi$ is instead $P$-ordinary for some parabolic subgroup $P$ of $G$, then $\pi_p$ no longer has trivial supercuspidal support. Furthermore, the analogous $P$-Iwahori subgroups now act on spaces of $P$-ordinary vectors via smooth finite-dimensional representations of $L_P$, a Levi factor of $P$. In \cite{Mar23a}, the author associates canonically such a representation $\tau$ to $\pi_p$, using the theory of Bushnell-Kutzko types \cite{BusKut98} and results of \cite{Pas05}, and construct $P$-ordinary vectors of \emph{type} $\tau$.

Therefore, our approach generalizes the construction of \cite{EHLS} by replacing $\mu$ with a general type $\tau$. In the process, we deal with the several challenges related to the dimension of the latter. In particular, one novelty of this paper is the construction of local Siegel-Weil sections (involved in the doubling method) associated to such types. Then, the main accomplishment of this paper is the computation of $p$-adic zeta integrals corresponding to such Siegel-Weil sections and the $P$-ordinary vectors of type $\tau$ mentioned above.

When $\tau$ is a character, we recover results of \cite{EHLS}. However, the author hopes that some of the techniques used here, using the point of view of Bushnell-Kutzko types and covers in the sense of \cite{BusKut98, BusKut99}, help to simplify and motivate various definitions that may seem \emph{ad hoc} in the ordinary setting. Our main result is Theorem \ref{main local thm}, which roughly says the following : 

\begin{theorem*}[Main Local Theorem]
    Let $\varphi_p \in \pi_p$ and $\cg{\varphi}_p \in \cg{\pi}_p$ be the $P$-anti-ordinary test vectors, of type $\tau$ and $\cg{\tau}$ respectively, defined in \eqref{def test vec varphi p and cg varphi p}. Let $\chi$ be a unitary Hecke character of $\KK$, $\chi_p = \otimes_{w \mid p} \chi_w$, and let $s \in \CC$. Let $f_{\tau, \cg{\tau}}^+ \in I_p(\chi_p, s)$ be the local Siegel-Weil section defined in Section \ref{construction f+}. 
    
    Then, the $p$-adic local zeta integral $I_p(\varphi_p, \cg{\varphi}_p, f_{\tau, \cg{\tau}}^+; \chi_p, s)$ defined in \eqref{def local zeta integrals} is equal to
    \begin{equation} \label{main eq intro}
        E_p\left( 
            s+\frac{1}{2}, \Pord, \pi_p, \chi_p
        \right)
            \cdot
        \frac{
            \vol(I_{P, r}^0)
            \vol(\tp{I}_{P, r}^0)
        }{
            \vol(I_{P, r}^0 \cap 
            \tp{I}_{P, r}^0)
        }
    \end{equation}
    where $E_p$ is a modified Euler factor at $p$, defined in Section \ref{section main local thm}.
\end{theorem*}

\subsubsection*{Additional comments.}
As the statement of our main theorem eludes, our result actually involves $P$-anti-ordinary vectors as opposed to $P$-ordinary vectors. This is simply due to the duality relations between Eisenstein series and automorphic representations in the doubling method. Relevant comparisons for these two notions are considered in \cite{Mar23a}. 

Our approach is actually formulated using vectors $\varphi_p$ and $\cg{\varphi}_p$ whose behavior under $P$-Iwahori subgroups are similar to the ones of $P$-anti-ordinary vectors, see \cite[Theorem 4.3, Lemma 4.6]{Mar23a}. However, we never explicitly assume that either of these vectors is $P$-(anti-)ordinary. The author hopes that by proceeding this way, our treatment is general enough to also be applied in other settings not considered in \cite{Mar23a}. 

For instance, throughout this paper, we assume a certain splitting condition on the prime $p$. However, when removing this condition, one works with the analogous notion of $\mu$-ordinary representations, in the sense of \cite{EisMan21}. We hope that our approach may be adapted in the future to this setting by replacing the theory of types and covers for general linear groups over local fields by the analogous theory for other classical groups (perhaps using results of \cite{MS}).

Furthermore, as explained in \cite[Remark 9.3.4]{EHLS}, the analogous computation in the literature of local zeta integrals at finite places away from $p$ over which $\pi$ ramifies is still unsatisfactory. The current approach is to choose local Siegel-Weil sections and test vectors at such primes so that the corresponding integral yields a constant volume factor. However, one could potentially use techniques developped in this paper to replace those non-optimal choices and use the doubling method to obtain better local zeta integrals related to Euler factors at these finite places of $L$-functions.

\subsubsection*{Structure of this paper.} 

In Section \ref{not and conv}, we first set some notation and introduce underlying assumptions for our work. We also review the theory of Bushnell-Kutzko types relevant for us. 

In Section \ref{P Iwa section}, we introduce the similitude unitary group $G$ on which we work and a certain parabolic subgroup $P$ of $G$. We then define level subgroups of $G(\ZZ_p)$ that are $P$-Iwahoric (of some level $r$). These two sections share many similarities with the set up of \cite{Mar23a}. We then conclude Section \ref{P Iwa section} by constructing specific vectors $\varphi_p$ and $\cg{\varphi}_p$ in $p$-factors of automorphic representations whose action under $P$-Iwahoric subgroups are described by certain types $\tau$ and $\cg{\tau}$. Our choices are inspired by the structure theorems for $P$-(anti-)ordinary representations proved in \cite{Mar23a}.

In Section \ref{doubling method}, we recall the set up of the doubling method for unitary groups and introduce the relevant local zeta integral. We then focus on the factor at $p$ and construct the local Siegel-Weil section $f_{\tau, \cg{\tau}}^+$ mentioned in the main theorem above. In doing so, one novelty of this paper is to extend matrix coefficients related to $\tau$ and $\cg{\tau}$ (defined as locally constant functions of $L_P(\ZZ_p)$) to locally constant functions supported on much larger open compact subsets of $G(\ZZ_p)$. This necessary step is already present in \cite[Section 4.3.1]{EHLS} and plays a crucial role to obtain the volume factors in \eqref{main eq intro}. To do so, they introduce a certain ``telescoping'' product (see \cite[p.56]{EHLS}). Although the formula is not complicated, it crucially relies on the fact the the analogue of $\tau$ is a character. Our approach generalizes their extension for $\tau$ of any dimension. We hope that our treatment, from the point of view of Bushnell-Kutzko types and covers, simplifies and motivates this step (see Equations \eqref{def extn mu aw} and \eqref{def ext mu w} below).

In Section \ref{comp above p}, we finally write down the expression for the $p$-adic zeta integral associated to $\varphi_p$, $\cg{\varphi}_p$ and $f^+_{\tau, \cg{\tau}}$. We then proceed to simplify it by obtaining as many cancellations as possible until we can finally use the Godement-Jacquet functional equation of \cite{Jac79}.

\subsubsection*{Acknowledgments} This paper contains the core computations necessary in the second third of the author's thesis supervised by Michael Harris. I am profoundly grateful for his many insightful comments to approach this problem and the necessary computations. In particular, I wish to thank Harris for his suggestions to look at the theory of Bushnell-Kutzko types and attempt to expose their usefulness to compute local zeta integrals.

I also wish to thank both Ellen Eischen and Christopher Skinner for helpful discussions where they provided encouragements and answers to my questions.

Lastly, many steps in the argument of Sections \ref{conv test vectors}, \ref{SW section at p} and \ref{comp of Zw} rely upon ideas found in the ordinary settings, most notably in \cite{EHLS}. It would have been far more difficult for me to carry out the simplifications in my calculations without the precise details provided by the four authors of this papers.

\input{S1/1-1.tex}
\input{S1/1-2.tex}
\input{S1/1-3.tex}

\input{S2/2-1.tex}
\input{S2/2-2.tex}
\input{S2/2-3.tex}

\input{S3/3-1.tex}
\input{S3/3-2.tex}

\input{S4/4-1.tex}

\bibliography{references.bib}
\bibliographystyle{amsalpha}

\end{document}

%% file: S1/1-1.tex
\section{Notation and conventions.} \label{not and conv}
Let $\bQQ \subset \CC$ be the algebraic closure of $\QQ$ in $\CC$. For any number field $F \subset \bQQ$, let $\Sigma_F$ denote its set of complex embedding $\hom(F, \CC) = \hom(F, \bQQ)$.

Throughout this article, we fix a CM field $\KK \subset \bQQ$ with ring of integers $\OO = \OO_{\KK}$. Let $\KK^+$ be the maximal real subfield of $\KK$ and denote its ring of integers as $\OO^+ = \OO_{\KK^+}$. Let $c \in \Gal(\KK/\KK^+)$ denote complex conjugation, the unique nontrivial automorphism. Given a place $v$ of $\KK$, we usually denote $c(v)$ as $\bar{v}$.

\subsection{CM types and local places.} \label{CM types}
Fix an integer prime $p$ that is unramified in $\KK$. Throughout this paper, we assume the following :

\begin{hypothesis}\label{above p split}
	Each place $v^+$ of $\KK^+$ above $p$ totally split as $v^+ = v \bar{v}$ in $K$.
\end{hypothesis}

Fix an algebraic closure $\bQQ_p$ of $\QQ_p$ and an embedding $\incl_p : \bQQ \hookrightarrow \bQQ_p$. Define
\[
	\bZZ_{(p)} = \{z \in \bQQ : \v_p(\incl_p(z)) \geq 0 \} \ ,
\]
where $\v_p$ is the canonical extension to $\bQQ_p$ of the normalized $p$-adic valuation on $\QQ_p$. 

Let $\CC_p$ be the completion of $\bQQ_p$. The map $\incl_p$ yields an isomorphism between its valuation ring $\OO_{\CC_p}$ and the completion of $\bZZ_{(p)}$ which extends to an isomorphism $\i : \CC \xrightarrow{\sim} \CC_p$.

Fix an embedding $\i_\infty : \QQ \hookrightarrow \CC$ such that $\incl_p = \i \circ \i_\infty$ and identify $\bQQ$ with its image in both $\CC$ and $\CC_p$.

Given $\sigma \in \Sigma_{\KK}$, the embedding $\incl_p \circ \sigma$ determines a prime ideal $\p_\sigma$ of $\Sigma_\KK$. There may be several embeddings inducing the same prime ideal. Similarly, given a place $w$ of $\KK$, let $\p_w$ denote the corresponding prime ideal of $\OO$.

Under Hypotesis \ref{above p split}, for each place of $\KK^+$ above $p$, there are exactly two primes of $\OO$ are both above it. Fix a set $\Sigma_p$ containing exactly one of these prime ideals for each place of $\KK^+$ above $p$. Moreover, let $\Sigma = \{\sigma \in \Sigma_\KK \mid \p_\sigma \in \Sigma_p\}$, a CM type of $\KK$ (see \cite[p.202]{Kat78}).

%% file: S1/1-2.tex

%% file: S1/1-3.tex
\subsection{Bushnell-Kutzko Types.} \label{BK types}
To discuss the local theory of $P$-ordinary representations in later sections, let us recall the theory of Bushnell-Kutzko types and covers, adapting the notions of \cite{BusKut98} and \cite[Section 3]{Lat21} to our setting.

Let $F = \KK_w$ for some place $w$ of $\OO_\KK$ and write $\OO_F$ for $\OO_{\KK_w}$. Let $G = \GL_n(F)$ for some integer $n$.

\subsubsection{Parabolic inductions.}
For any parabolic subgroup $P$ of $G$, let $L$ and $P^u$ be a Levi factor and its unipotent radical, respectively. Let $\delta_P : P \to \CC^\times$ denote its modulus character.

Recall that $\delta_P$ factors through $L$. Moreover, if $P$ is the standard parabolic subgroup associated to the partition $n = n_1 + \ldots + n_s$, one has
\begin{equation}
	\delta_P(l) = \prod_{k=1, \ldots, s} \absv{\det(l_k)}^{-\sum_{i < k} n_i + \sum_{j > k} n_j}
\end{equation}
for any $l = (l_1, \ldots, l_s)$ in $L = \prod_{k=1}^s \GL_{n_k}(F)$.

Given a smooth representation $\sigma$ of $L$, we often consider $\sigma$ as a representation of $P$ without comments. Let $\Ind_P^{G} \sigma $ denote the classical parabolic induction functor from $P$ to $G$.  Similarly, the \emph{normalized} parabolic induction functor is
\[
    \ind{P}{G} \sigma = \Ind_P^{G} (\sigma \otimes \delta_P^{1/2})
\]

In Sections \ref{doubling method} and \ref{comp above p}, we prefer to work with the normalized version but the main computations can entirely be done with unnormalized parabolic induction as well.

\subsubsection{Supercuspidal support} \label{sc support}
A theorem of Jacquet (see \cite[Theorem 5.1.2]{Cas95}) implies that given any irreducible representation $\pi$ of $G$, one may find a parabolic subgroup $P$ of $G$ with Levi subgroup $L$ and a supercuspidal representation $\sigma$ of $L$ such that $\pi \subset \ind{P}{G} \sigma$. .

The pair $(L, \sigma)$ is uniquely determined by $\pi$, up to $G$-conjugacy and one refers to this conjugacy class as the \emph{supercuspidal support} of $\pi$.

Consider two pairs $(L, \sigma)$ and $(L', \sigma')$ consisting of a Levi subgroup of $G$ and one of its supercuspidal representation. One says that they are \emph{$G$-inertially equivalent} if there exists some $g \in G$ such that $L' = g^{-1}Lg$ and some unramified character $\chi$ of $L'$ such that $\prescript{g}{}{\sigma} \cong \sigma' \otimes \chi$, where $\prescript{g}{}{\sigma}(x) = \sigma(gxg^{-1})$. We write $[L, \sigma]_G$ for the $G$-inertial equivalence class of $(L, \sigma)$.

For such a class $\s$, let $\Rep^\s(G)$ denote the full subcategory of $\Rep(G)$ whose objects are the representations such that all their irreducible subquotients have inertial equivalence class $\s$. The Bernstein-Zelevinsky geometric lemma (see \cite[Section VI.5.1]{Ren10}) implies that $\ind{P}{G} \sigma \in \Rep^\s(G)$, where $\s = [L, \sigma]_G$.

\begin{definition}[\cite{BusKut98}]
    Let $J$ be a compact open subgroup of $G$ and $\tau$ be an irreducible represention of $J$. Let $\Rep_\tau(G)$ denote the full subcategory of $\Rep(G)$ whose objects are the representations generated over $G$ by their $\tau$-isotypic subspace. We say that $(J, \tau)$ is an $\s$-type if $\Rep_\tau(G) = \Rep^\s(G)$. 
\end{definition}

Let $\s_L = [L, \sigma]_L$. It follows from \cite[Theorem 1.3]{Pas05} that there exists a unique (up to isomorphism) representation $\tau$ of $K = G(\OO_F)$ such that $(K, \tau)$ is an $\s$-type. We refer to this unique ``maximal'' type of $\s_L$ as the \emph{BK-type} of $\pi$.

\subsubsection{Covers of types.}
Fix a Levi subgroup $L$ of $G$ as well as an irreducible supercuspidal representation $\sigma$ of $L$. Let $\s_L = [L, \sigma]_L$ denote its inertial support and let $(J, \tau)$ be an $\s_L$-type.

\begin{definition}[\cite{BusKut98}] \label{Def G cover}
    Let $\JJ$ be a compact open subgroup of $G$ and $\uptau$ be an irreducible representation of $\JJ$. We say that $(\JJ, \uptau)$ is a \emph{$G$-cover} (or a \emph{cover to $G$)} of $(J, \tau)$ if all of the following properties are satisfied : 
    \begin{enumerate}
        \item[(i)] $\JJ \cap L = J$
        \item[(ii)] $\uptau|_{J} = \tau$. In particular, both $\uptau$ and $\tau$ act on the same vector space.
        \item[(iii)] Let $P$ be any parabolic of $G$ with Levi factor $L$. Write $P^u$ for its unipotent radical and $P^l$ for the opposite of $P^u$. Let $\JJ^l = \JJ \cap P^l$ and $\JJ^u = \JJ \cap P^u$. Then, $\JJ = \JJ^l J \JJ^u$.
        \item[(iv)] The kernel of $\uptau$ contains both $\JJ^l$ and $\JJ^u$.
        \item[(v)] For any smooth representation $(\Pi, \mathcal{V})$ of $G$, the natural $L$-homomorphism $\pr_P : \mathcal{V} \to \mathcal{V}_P$ restricts to an injection on the $\uptau$-isotypic subspace of $\Pi$.
    \end{enumerate}
\end{definition}

The main result of \cite[Theorem 3.10]{Mar23a} is concerned with the construction of a canonical cover of the BK-type $\tau$ of $\pi$ inside the vector space associated to $\pi$, see \cite[Remark 3.12]{Mar23a}. Such covers are constructed in various settings of $P$-(anti-)ordinary representations. Their properties guide our choice of test vectors in Section \ref{conv test vectors}.

%% file: S2/2-1.tex
\section{$P$-Iwahoric level subgroups of unitary groups.} \label{P Iwa section}

\subsection{Unitary Groups} \label{unitary groups} 
Let $V$ be a finite-dimensional $\KK$-vector space, equipped with a pairing $\brkt{\cdot}{\cdot}_{V}$ that is Hermitian with respect to the quadratic extension $\KK/\KK^+$. Write $n = \dim_\KK V$.

\subsubsection{PEL datum of unitary type} \label{PEL unitary}
Let $\delta \in \OO$ be totally imaginary and prime to $p$ and define $\brkt{\cdot}{\cdot} = \tr_{\KK / \QQ}(\delta \brkt{\cdot}{\cdot}_{V})$. This choice of $\delta$ and our Hypothesis \eqref{above p split} ensure the existence of an $\OO$-lattice $L \subset V$ such that the restriction of $\brkt{\cdot}{\cdot}$ to $L$ is integral and yields a perfect pairing on $L \otimes \ZZ_p$.

For each $\sigma \in \Sigma_{\KK}$, let $V_{\sigma}$ denote $V \otimes_{\KK, \sigma} \CC$. It has a $\CC$-basis diagonalizing the pairing $\brkt{\cdot}{\cdot}$. The only eigenvalues must be $\pm 1$, say that $1$ (resp. $-1$) has multiplicity $r_{\sigma}$ (resp. $s_{\sigma}$). We order the basis so that the $+1$-eigenvectors appear first. Fixing such a basis, let $h_{\sigma} : \CC \to \End_{\RR}(V_{\sigma})$ be $h_{\sigma} = \diag(z 1_{r_{\sigma}}, \bar{z} 1_{s_{\sigma}})$.

Let $h = \prod_{\sigma \in \Sigma} h_{\sigma} : \CC \to \End_{\KK^+ \otimes \RR}(V \otimes \RR)$, using the canonical identification
\[
    \prod_{\sigma \in \Sigma}
        \End_{\RR}(V_{\sigma}) 
    = 
        \End_{\KK^+ \otimes \RR}(V \otimes \RR)
\]
provided by our fixed choice of CM type $\Sigma$ of $\KK$.

The tuple 
$
    \PP = 
        (
            \KK, c, \OO, L, 
            2\pi\sqrt{-1}\brkt{\cdot}{\cdot}, h
        )
$ 
is a PEL datum of unitary type, as defined in \cite[Section 2.1-2.2]{EHLS}. It has an associate group scheme $G = G_\PP$ over $\ZZ$ whose $R$-points are
\begin{equation} \label{def sim uni gp}
	G(R) = \{ (g, \nu) \in \GL_{\OO \otimes R}(L \otimes R) \times R^\times \mid \brkt{gx}{gy} = \nu \brkt{x}{y}, \forall x, y \in L \otimes R \},
\end{equation}
for any commutative ring $R$. In particular, $G_{/\QQ}$ is a reductive group. Moreover, the assumptions on $p$ imply that $G_{/\ZZ_p}$ is smooth and $G(\ZZ_p)$ is a hyperspecial maximal compact of $G(\QQ_p)$.

\subsubsection{Ordinary hypothesis on signature} \label{Hodge structure}
The homomorphism $h$ determines a pure Hodge structure of weight $-1$ on $V_\CC = L \otimes \CC$, i.e. $V = V^{-1, 0} \oplus V^{0,-1}$ and $h(z)$ acts as $z$ on $V^{-1,0}$ and as $\bar{z}$ on $V^{0,-1}$. In particular, the $\OO \otimes \CC$-submodule $V^0 \subset V$ defined as the degree 0 piece of the corresponding Hodge filtration is simply $V^{-1,0}$. 

For each $\sigma \in \Sigma_{\KK}$, let $a_{\sigma} = \dim_{\CC} (V^0 \otimes_{\OO \otimes \CC, \sigma} \CC)$ and $b_{\sigma} = n - a_{\sigma}$. The \emph{signature} of $h$ is defined as the collection of pairs $\{ (a_{\sigma}, b_{\sigma} )_{\sigma \in \Sigma_\KK} \}$. In fact, $(a_{\sigma}, b_{\sigma}) = (r_{\sigma}, s_{\sigma})$ if $\sigma \in \Sigma$. Otherwise, one has $(a_{\sigma}, b_{\sigma}) = (s_{\sigma}, r_{\sigma})$.

\begin{hypothesis}[Ordinary hypothesis]
	For all embeddings $\sigma, \sigma' \in \Sigma_{\KK}$, if $\p_{\sigma} = \p_{\sigma'}$, then $a_{\sigma} = a_{\sigma'}$.
\end{hypothesis}

Throughout this paper, we assume Hypothesis 2.1. Therefore, given a place $w$ of $\KK$ above $p$, one can define $(a_{w}, b_{w}) := (a_{\sigma}, b_{\sigma})$, where $\sigma \in \Sigma_\KK$ is any embedding such that $\p_\sigma = \p_w$.

%% file: S2/2-2.tex
\subsection{Structure of $G$ over $\ZZ_p$.}
\subsubsection{Comparison to general linear groups.} \label{G(Zp) to GLn}
The factorization $\OO \otimes \ZZ_p = \prod_{w \mid p} \OO_{w}$, over primes $w$ of $\KK$ above $p$, yields a decomposition $L \otimes \ZZ_p = \prod_{w \mid p} L_w$. It corresponds to 
\begin{equation}
    \GL_{
        \OO \otimes \ZZ_p
    }(L \otimes \ZZ_p) 
        \xrightarrow{\sim} 
    \prod_{w \mid p} 
        \GL_{\OO_w}(L_w), 
            \ \ \ \ g \mapsto (g_{w}) \ ,
\end{equation}
a canonical $\ZZ_p$-isomorphism. From the above, one obtains the identification
\begin{equation} \label{prod G over Zp}
    G_{/\ZZ_p} 
        \xrightarrow{\sim} 
    \Gm \times 
    \prod_{w \in \Sigma_p} 
        \GL_{\OO_w}(L_w), 
            \ \ \ \ (g, \nu) \mapsto (\nu, (g_{w})) \ .
\end{equation}

Furthermore, our assumption above about the pairing $\brkt{\cdot}{\cdot}$ implies that for each $w \mid p$, there is an $\OO_w$-decomposition of $L_w = L_w^+ \oplus L_w^-$ such that 
\begin{enumerate}
	\item $\rk_{\OO_w}{L_w^+} = a_{w}$ and $\rk_{\OO_w}{L_w^-} = b_{w}$;
	\item Upon restricting $\brkt{\cdot}{\cdot}$ to $L_w \times L_{\bar{w}}$, the annihilator of $L_w^\pm$ is $L_{\bar{w}}^{\pm}$. Hence, one has a perfect pairing $L_w^+ \oplus L_{\bar{w}}^- \to \ZZ_p(1)$, again denoted $\brkt{\cdot}{\cdot}$.
\end{enumerate}

Fix dual $\OO_w$-bases (with respect to the perfect pairing above) for $L_w^+$ and $L_{\bar{w}}^-$. They yield identifications
\begin{equation} \label{GL(ei Lw) basis}
    \begin{tikzcd}
            \GL_{\OO_w}(L_w^+) 
                \arrow[r, "\cong"] 
        &
            \GL_{a_{w}}(\OO_w) 
        &
            \GL_{b_{\bar{w}}}(\OO_{\bar{w}}) 
                \arrow[r, "\cong"] 
        &
            \GL_{\OO_w}(L_{\bar{w}}^-)
    \end{tikzcd}
\end{equation}
as well as an isomorphism $\GL_{\OO_w}(L_w) \cong \GL_{n}(\OO_w)$ such that the obvious map
\[
	\GL_{\OO_w}(L_w^+) \times \GL_{\OO_w}(L_w^-) \hookrightarrow \GL_{\OO_w}(L_w)
\]
is simply the diagonal embedding of block matrices.

Let $L^\pm = \prod_{w \mid p} L_w^\pm$ and set $H := \GL_{\OO \otimes \ZZ_p}(L^+)$. Then, the identification \eqref{GL(ei Lw) basis} above induces a canonical isomorphism
\begin{equation} \label{def H}
    H \cong
    \prod_{w \mid p}
        \GL_{a_{w}}(\OO_w) = 
    \prod_{w \in \Sigma_p}
        \GL_{a_{w}}(\OO_w) \times 
        \GL_{b_{w}}(\OO_w)
\end{equation}

\subsubsection{Parabolic subgroups of $G$ over $\ZZ_p$.} \label{level at p}
For $w \mid p$, let 
\[
	\d_{w} = \left( n_{w,1}, \ldots, n_{w, t_{w}} \right)
\]
be a partition of $a_{w} = b_{\bar{w}}$. Let $P_{\d_{w}} \subset \GL_{a_{w}}(\OO_w)$ denote the standard parabolic subgroup corresponding to $\d_{w}$. Define $P_H \subset H$ as the $\ZZ_p$-parabolic that corresponds to the products of all the $P_{\d_{w}}$ via the isomorphism \eqref{def H}. We denote the unipotent radical of $P_H$ by $P_H^u$.

We identify the elements of the Levi factor $L_H = P_H / P_H^u$ of $P_H$ with collections of block-diagonal matrices, with respect to the partitions $\d_{w}$, via \eqref{def H}. 

Let $P^+ \subset G_{/\ZZ_p}$ be the parabolic subgroup that stabilizes $L^+$ and such that
\begin{equation} \label{def P+}
    P^+ 
        \twoheadrightarrow 
    \Gm \times P_H \subset \Gm \times H
\end{equation}
where the map to the first factor is the similitude character $\nu$ and the map to the second factor is projection to $H$. 

For $w \in \Sigma_p$, let $P_{w}$ be the parabolic subgroup of $\GL_{\OO_w}(L_w)$ given by
\begin{equation} \label{local parabolics}
    P_{w} = 
    \left\{
        \begin{pmatrix}
		A & B \\
		0 & D
        \end{pmatrix} 
            \in \GL_{n}(\OO_w) \mid
                A \in P_{\d_w}, 
                D \in P^{\opp}_{\d_{\ol{w}}}
    \right\} \ ,
\end{equation}
via the isomorphism \eqref{GL(ei Lw) basis}.

We identify $P = \prod_{w \in \Sigma_p} P_{w}$ as a subgroup of $G_{/\ZZ_p}$ via \eqref{prod G over Zp}.

\begin{remark} \label{trivial partition}
    The trivial partition of $a_{w}$ is $(1, \ldots, 1)$ (of length $t_{w} = a_{w}$). If the partitions $\d_{w}$ and $\d_{\bar{w}}$ are both trivial, we write $B_{w}$ instead of $P_{w}$. In that case, $L_B$ is the maximal torus subgroup of $\GL_n(\OO_w)$.
\end{remark}

Our choices of bases above imply that under the isomorphisms \eqref{prod G over Zp} and \eqref{GL(ei Lw) basis}, $P^+$ corresponds to
\begin{equation}
	P^+ \xlongrightarrow{\sim} \Gm \times P \ .
\end{equation}

In particular, this induces a natural isomorphism $L_H \cong L_P := P/P^u$, where $P^u$ is the unipotent radical of $P$.

\begin{definition} \label{def PIwahori}
    We define the $P$-Iwahori subgroup of $G$ of level $r \geq 0$ as
    \[
        I_r^0 = I_{P,r}^0 :=
        \left\{
            g \in G(\ZZ_p) \mid 
                g \text{ mod } p^r 
                \in P^+(\ZZ_p/p^r \ZZ_p)
        \right\}
    \]
    and the pro-$p$ $P$-Iwahori subgroup $I_r = I_{P,r}$ of $G$ of level $r$ as 
    \[
        I_r = I_{P,r} :=
        \left\{
            g \in G(\ZZ_p) \mid 
                g \text{ mod } p^r 
                \in (\ZZ_p/p^r\ZZ_p)^\times 
                \times P^u(\ZZ_p/p^r \ZZ_p)
        \right\}.
    \]
\end{definition}

\begin{remark}
    We refrain from referring to $I_r^0$ as a \emph{parahoric} subgroup of $G$. This terminology is usually reserved for stabilizers of points in Bruhat-Tits building. We make no attempt here to introduce our construction from the point of view of these combinatorial and geometric structures.
\end{remark}

The inclusion of $L_P(\ZZ_p)$ in $I_r^0$ yields a canonical isomorphism 
\begin{equation}
    L_P(\ZZ_p/p^r\ZZ_p) \xrightarrow{\sim} I_r^0/I_r \ .
\end{equation}

For each $w \in \Sigma_p$, one similarly defines $I_{w, r}^0$ and $I_{w, r}$ by replacing $P^+$ by $P_{w}$ and working in $\GL_{n}(\OO_w)$ instead of $G(\ZZ_p)$. Let
\[
    I_r^{\GL} = 
    \prod_{w \in \Sigma_p}
        I_{w, r} 
        \ \ \ \text{and} \ \ \ 
    I_r^{0,\GL} = 
    \prod_{w \in \Sigma_p} 
        I_{w, r}^0 \ ,
\]
so that $I_{r}$ and $I_{r}^0$ correspond to $\ZZ_p^\times \times I_{P,r}^{\GL}$ and $\ZZ_p^\times \times I_{P,r}^{0, \GL}$ respectively, via the isomorphisms \eqref{prod G over Zp} and \eqref{GL(ei Lw) basis}.

\subsubsection{Opposite unitary groups} \label{def G1 and G2}
In Section \ref{conv test vectors} below, we work with automorphic representations of $G(\AA)$ as well as automorphic representations on the opposite unitary group. We recall this notion and the notation of \cite[Section 3]{EHLS} relevant for our purpose.

Let
$
    \PP_1 := \PP = 
        (
            \KK, c, \OO, L, 
            2\pi\sqrt{-1}\brkt{\cdot}{\cdot}, h
        )
$ 
be the PEL datum constructed in Section \ref{PEL unitary} and define
$
    \PP_2 := 
        (
            \KK, c, \OO, L, 
            -2\pi\sqrt{-1}\brkt{\cdot}{\cdot}, h(\bar{\cdot})
        )
$. In particular, the signature at $w \in \Sigma_p$ of $\PP_2$ is $(b_w, a_w) = (a_{\ol{w}}, b_{\ol{w}})$. We typically write $G_i$ for the unitary group corresponding to $\PP_i$ when we want to distinguish the underlying PEL datum.

Although there is a canonical identification $G_1(\AA) = G_2(\AA)$, we distinguish their level structure at $p$. Namely, if $L_i$ is the lattice associated to $\PP_i$, we always fix an $\OO \otimes \ZZ_p$-decomposition $L_i \otimes \ZZ_p = L_i^+ \oplus L_i^-$. For $i=1$, we choose $L_1^\pm = L^\pm$, using the notation of Section \ref{G(Zp) to GLn}. For $i=2$, we instead set $L_2^\pm = L^\mp$. We have a corresponding factorization $L_i^\pm = \prod_{w \mid p} L_{i, w}^\pm$. Effectively, for each $w \in \Sigma_p$, this provides two distinct identifications of $\GL_{\OO_w}(L_{i,w})$ with $\GL_n(\OO_w)$ via \eqref{GL(ei Lw) basis}.

However, we prefer to only work with this identification for $i=1$. In other words, for $i=1,2$, let $P_{i, w}$ be the local parabolic subgroup of $\GL_{\OO_w}(L_{i,w})$ defined in Equation \eqref{local parabolics}. Then, via the identification \eqref{GL(ei Lw) basis} for $i=1$, the parabolic subgroup $P_{2,w}$ of $\GL_n(\OO_w)$ corresponds to $\tp{P}_{1,w}$. Therefore, in what follows we always work with $P_w := P_{1,w}$ when consider $G_1$ (or equivalently, the PEL datum $\PP_1$) and with $\tp{P}_w$ when considering $G_2$.

Similarly, the local Iwahori subgroups of level $r$ defined below Definition \ref{def PIwahori} are $I_{w, r} \subset I_{w, r}^0$ when working with $G_1$ and their transpose when working with $G_2$.

%% file: S2/2-3.tex
\subsection{Conventions on test vectors.} \label{conv test vectors}
In this section, we make several choices of compatible $P$-anti-ordinary vectors in automorphic representations, using the theory of Bushnell-Kutzko types \cite{BusKut98} and the main results of \cite{Mar23a}.

\subsubsection{Local representations over CM type at $p$.}
Let $\pi$ be a cuspidal (irreducible) automorphic representation of $G(\AA)$. By this, we mean an irreducible $(\g, U_\infty) \times G(\AA_f)$-submodule of the space of cuspidal automorphic forms on $G(\AA)$, where $\g = \Lie(G(\RR))_\CC$ and $U_\infty \subset G(\RR)$ is the stabilizer of the morphism $h$ provided in the PEL datum $\PP$ associated to $G$.

For each rational prime $l$, write $G_l := G(\QQ_l)$. Consider the restricted products $G(\AA_f) = \prod_l' G_l$ and $G(\AA) = G_\infty \times \prod_l' G_l$, and let $\pi_l$ be the $l$-constituent of $\pi$.

Let $S$ be a finite set of places of $\QQ$ containing all primes such that $\pi_l$ is ramified and assume that $p \notin S$. For each $l \notin S \cup \{p\}$, fix a nonzero unramified vector $\varphi_{l,0} \in \pi_l$. Correspondingly, fix isomorphisms
\begin{equation} \label{facto pi}
    \pi 
        \xrightarrow{\sim} 
            \pi_\infty \otimes \pi_f 
        \ \ \ ; \ \ \ 
    \pi_f
        \xrightarrow{\sim} 
            \pi^{S, p} \otimes 
            \pi_p \otimes
            \pi_{S} \ ,
\end{equation}
where
\[
    \pi_S \cong \bigotimes_{l \in S} \pi_l
        \ \ \ ; \ \ \
    \pi^{S, p} \cong \bigotimes_{l \notin S \cup \{p\}} \pi_l
\]
and the second factorization is a restricted tensor products with respect to our choice of unramified vectors $\varphi_{l,0}$. 

Let $G_w := \GL_n(\KK_w)$. The isomorphisms \eqref{prod G over Zp} and \eqref{GL(ei Lw) basis} induce an identification
\begin{equation} \label{facto pi p}
    \pi_p \cong \mu_p \otimes \left( \bigotimes_{w \in \Sigma_p} \pi_w \right) \ ,
\end{equation}
for some character $\mu_p$ of $\QQ_p$ and irreducible admissible representation $\pi_w$ of $G_w$. 

Now, assume there exists an integer $r \gg 0$ such that $\pi_p^{I_{P, r}} \neq 0$. Equivalently, assume $\mu_p$ is unramified and
\[
    \pi_p^{I_{P, r}} \cong \bigotimes_{w \in \Sigma_p} \pi_w^{I_{w, r}} \neq 0 \ .
\]

Let $\cg{\pi}$ denote the contragredient of $\pi$. Since it is a twist of the complex conjugate of $\pi$, it is also a cuspidal automorphic representation of $G(\AA)$. We fix similar conventions for $\cg{\pi}$. In particular, we have
\begin{equation} \label{facto cg pi p}
    \cg{\pi}_p \cong \mu_p^{-1} \otimes \left( \bigotimes_{w \in \Sigma_p} \cg{\pi}_w \right)
\end{equation}
and $\left( \cg{\pi}_w \right)^{\tp{I_{w, r}}} \neq 0$ for each $w \in \Sigma_p$. Here, $\cg{(\cdot)}$ is again used to denote contragredient representations. We keep this convention throughout the rest of this article.

In what follows, we choose specific ``test'' vectors in $\pi_p$ and $\cg{\pi}_p$ using notions and results from \cite{Mar23a}. Then, in Section \ref{SW section at p}, we construct a local Siegel-Weil section so that particular $p$-adic zeta integrals defined in Section \ref{zeta integrals} can be related to $p$-Euler factors of standard automorphic $L$-functions.

\subsubsection{Compatibility of parabolic subgroups.} 
For each $w \in \Sigma_p$ and integer $d \geq 1$, let $G_w(d)$ denote the algebraic group $\GL(d)$ over $\OO_w = \OO_{\KK_w}$. However, when $d = n$, we still write $G_w$ instead of $G_w(n)$. Let $(a_w, b_w)$ be the signature associated to $w \in \Sigma_p$, $L_1$ and $\langle \cdot, \cdot \rangle_1$, as in Section \ref{Hodge structure}. 

Proceeding as in Section \ref{level at p}, let $P_{a_w} \subset G_w(a_w)$, $P_{b_w} \subset G_w(b_w)$ and $P_{a_w, b_w} \subset G_w$ be the standard upper triangular parabolic subgroups associated to partitions
\[
    \d_{a_w} = 
        (n_{w, 1}, \ldots, n_{w, t_w}) 
    \ \ \ ; \ \ \ 
    \d_{b_w} = 
        (n_{w, t_w+1}, \ldots, n_{w, r_w}) 
    \ \ \ ; \ \ \
    \d_{w} = 
        (a_w, b_w) 
\]
of $a_w$, $b_w$ and $n$, respectively. We also work with the parabolic subgroup $P_w \subset G_w$ constructed in Section \ref{level at p}. Note that $P_w \subset P_{a_w, b_w} \subset G_w$.

For any one of these parabolic subgroup $P_\bullet$, let $L_\bullet$ denote its standard Levi subgroup consisting of block-diagonal matrices (corresponding to the decomposition defining $P_\bullet$). Similarly, consider the pro-$p$ Iwahori subgroup $I_{\bullet, r}$ of level $r$ associated to $P_\bullet$ consisting of invertible matrices $g$ (of the appropriate size) over $\OO_w$ such that $g$ mod $\p^r$ is in $P_\bullet^u(\OO_w/\p^r\OO_w)$.

Let $K_\bullet = L_\bullet(\OO_w)$ and $I_{\bullet, r}^0 = K_\bullet I_{\bullet, r}$. Setting $K_{w,j} = \GL_{n_{w,j}}(\OO_w)$, we have
\[
    K_{a_w} = 
        \prod_{j=1}^{t_w}
            K_{w,j} \ \ \ ; \ \ \
    K_{b_w} = 
        \prod_{j=t_w+1}^{r_w}
            K_{w,j} \ \ \ ; \ \ \
    K_{w} = 
        K_{a_w} \times K_{b_w} \ ,
\]
where the products take place in $G_w(a_w)$, $G_w(b_w)$ and $G_w$, respectively.

\subsubsection{Compatibility of local representations} \label{compatibility sigmas}
Assume there exists an admissible irreducible representation $\sigma_w$ of $L_w$ such that $\pi_w$ is the unique irreducible quotient of $\ind{P_w}{G_w} \sigma_w$. Equivalently, $\cg{\pi}_w$ is the unique irreducible subrepresentation of $\ind{P_w}{G_w} \cg{\sigma}_w$.

\begin{remark}
    As explained in \cite[Section 4.1.1]{Mar23a}, an easy application of a theorem of Jacquet (see \cite[Lemma 3.6]{Mar23a} or \cite[Theorem 5.1.2]{Cas95}) implies that if $\pi_p$ is $P$-anti-ordinary, such admissible representations exist.
\end{remark}

Write $\sigma_w = \boxtimes_{j=1}^{r_w} \sigma_{w,j}$ and consider the representations  
\[
    \sigma_{a_w} = 
        \boxtimes_{j=1}^{t_w}
            \sigma_{w,j} \ \ \ ; \ \ \
    \sigma_{b_w} = 
        \boxtimes_{j=t_w+1}^{r_w}
            \sigma_{w,j}
\]
of $L_{a_w}$ and $L_{b_w}$. Let $\pi_{a_w}$ and $\pi_{b_w}$ be the unique irreducible quotients
\begin{equation} \label{def pi aw and pi bw}
    \ind{P_{a_w}}{G_w(a_w)} \sigma_{a_w} 
        \twoheadrightarrow \pi_{a_w}
    \ \ \ \text{and} \ \ \
    \ind{P^{\opp}_{b_w}}{G_w(b_w)} \sigma_{b_w} 
        \twoheadrightarrow \pi_{b_w} \ ,
\end{equation}
and set $\pi_{a_w, b_w} := \pi_{a_w} \boxtimes \pi_{b_w}$. Under the canonical isomorphism
\begin{equation}\label{iso induced rep}
    \ind{P_w}{G_w} \sigma_w 
        \xrightarrow{\sim}
    \ind{P_{a_w, b_w}}{G_w} 
    \left(
        \ind{
            P_{a_w} \times P_{b_w}^{\opp}
        }{
            G_w(a_w) \times G_w(b_w)
        }
        \sigma_{a_w} \boxtimes \sigma_{b_w}
    \right) \ ,
\end{equation}
given by $\phi \mapsto (g \mapsto (h \mapsto \phi(hg))$, $\pi_w$ is the unique irreducible quotient
\begin{equation} \label{pi aw x pi bw to pi w}
    \ind{P_{a_w,b_w}}{G_w}
    \left(
        \pi_{a_w, b_w} 
    \right)
        \twoheadrightarrow 
    \pi_w \ .
\end{equation}

\subsubsection{Conventions on pairings} \label{conv on pairings}
In this section, we set conventions on pairings between pairs of contragredient representations, as in \cite[Section 4.3.3]{EHLS}.

Let $\la \cdot, \cdot \ra_{\sigma_{w,j}}$ be the tautological pairing between $\sigma_{w,j}$ and its contragredient $\cg{\sigma}_{w,j}$. Then, define $(\cdot, \cdot)_{a_w} = \otimes_{i=1}^{t_w} \la \cdot, \cdot \ra_{\sigma_{w,j}}$ so that
\begin{align*}
    \la \cdot, \cdot \ra_{a_w} &: 
        \left( 
            \ind{P_{a_w}}{G_w(a_w)} 
                \sigma_{a_w}
        \right)
            \times 
        \left( 
            \ind{P_{a_w}}{G_w(a_w)} 
                \cg{\sigma}_{a_w}
        \right) 
    \to \CC \\
    \la 
        \varphi, 
        \cg{\varphi} 
    \ra_{a_w} 
    &= 
    \int_{K_{a_w}} 
        ( 
            \varphi(k), 
            \cg{\varphi}(k) 
        )_{a_w} 
    dk
\end{align*}
is the perfect $G_w(a_w)$-invariant pairing that identify the above pair as contragredient representations. A similar logic applies for $(\cdot, \cdot)_{b_w} = \otimes_{i=t_w+1}^{r_w} \la \cdot, \cdot \ra_{\sigma_{w,j}}$ and
\begin{align*}
    \la \cdot, \cdot \ra_{b_w} &: 
        \left( 
            \ind{P^{\opp}_{b_w}}{G_w(b_w)} 
                \sigma_{b_w}
        \right)
            \times 
        \left( 
            \ind{P^{\opp}_{b_w}}{G_w(b_w)} 
                \cg{\sigma}_{b_w}
        \right) 
    \to \CC \\
    \la 
        \varphi, 
        \cg{\varphi} 
    \ra_{b_w} 
    &= 
    \int_{K_{b_w}} 
        ( 
            \varphi(k), 
            \cg{\varphi}(k) 
        )_{b_w} 
    dk \ .
\end{align*}

Taking the dual of the surjections in Equation \eqref{def pi aw and pi bw} yields injections
\begin{equation} \label{def pi aw and pi bw dual}
    \cg{\pi}_{a_w} \hookrightarrow 
        \ind{P_{a_w}}{G(a_w)} 
            \cg{\sigma}_{a_w} 
    \ \ \ \text{ and } \ \ \ 
    \cg{\pi}_{b_w} \hookrightarrow
        \ind{P_{b_w}^{\opp}}{G(b_w)} 
            \cg{\sigma}_{b_w} 
\end{equation}
and restricting the second argument of $\brkt{\cdot}{\cdot}_{a_w}$ to $\cg{\pi}_{a_w}$ makes the first argument of the pairing factor through $\pi_{a_w}$. It is identified with the tautological pairing $\brkt{\cdot}{\cdot}_{\pi_{a_w}} : \pi_{a_w} \times \cg{\pi}_{a_w} \to \CC$. Again, a similar logic applies for $\brkt{\cdot}{\cdot}_{\pi_{b_w}} : \pi_{b_w} \times \cg{\pi}_{b_w} \to \CC$.

Let $(\cdot, \cdot)_w = \la \cdot, \cdot \ra_{\pi_{a_w}} \otimes \la \cdot, \cdot \ra_{\pi_{b_w}}$. As above, it determines a pairing
\[
    \la \cdot, \cdot \ra_w : 
        \ind{P_{a_w, b_w}}{G_w}  
            \left(
                \pi_{a_w, b_w}
            \right)
            \times 
        \ind{P_{a_w, b_w}}{G_w} 
            \left( 
                \cg{\pi}_{a_w, b_w}
            \right)
    \to \CC
\]
as well as a pairing $\brkt{\cdot}{\cdot}_w : \pi_w \times \cg{\pi}_w \to \CC$, using the dual 
$
    \cg{\pi}_w \hookrightarrow 
    \ind{P_{a_w, b_w}}{G_w} 
        \left( 
            \cg{\pi}_{a_w, b_w}
        \right)
$
induced from Equation \eqref{pi aw x pi bw to pi w}.

For any $\phi \in \pi_w$, $\cg{\phi} \in \cg{\pi}_w$, if $\varphi$ is a lift of $\phi$ and $\cg{\varphi}$ is the image of $\cg{\phi}$, then
\begin{equation} \label{inner product pi w integral}
    \brkt{\phi}{\phi}_{\pi_w} =
    \int_{\GL_n(\OO_w)}
        \left(
            \varphi(k),
            \cg{\varphi}(k)
        \right)_w
    dk
\end{equation}

\subsubsection{Compatibility of test vectors.} \label{comp test vectors}
Fix two characters $\chi_{w,1}, \chi_{w,2} : \KK_w^\times \to \CC^\times$. Choose any integer $r$ such that
\begin{equation} \label{large enough r}
	r \geq 
	\max
	(
		1, 
		\ord_w(\cond(\chi_{w,1})) , 
		\ord_w(\cond(\chi_{w,2}))
	)
\end{equation}

In what follows, we consider $\chi_{w,1}$ and $\chi_{w,2}$ as characters of general linear groups of any rank via composition with the determinant without comment.

Furthermore, for each $1 \leq j \leq r_w$, let $\tau_{w,j}$ be a smooth (finite-dimensional) irreducible representation of $K_{w,j}$. We assume that $r$ is large enough so that $\tau_{w,j}$ factors through $\GL_{n_{w,j}}(\OO_w/\p_w^r\OO_w)$. Assume there exists an embedding $\alpha_{w,j}$ of $\tau_{w,j}$ in the restriction of $\sigma_{w,j}$ as a representation of $K_{w,j}$. 

Let $\alpha_{a_w} : \tau_{a_w} \to \sigma_{a_w}$ and $\alpha_{b_w} : \tau_{b_w} \to \sigma_{b_w}$ be the corresponding embeddings over $K_{a_w}$ and $K_{b_w}$ respectively, where
\[
    \tau_{a_w} = 
        \boxtimes_{j=1}^{t_w}
            \tau_{w,j} \ \ \ ; \ \ \
    \tau_{b_w} = 
        \boxtimes_{j=t_w+1}^{r_w}
            \tau_{w,j} \ .
\]

\begin{remark} \label{tau BK type sigma}
    Implicitly, we think of $\tau_{a_w}$ as a Bushnell-Kutzko type for the $\sigma_{a_w}$ in the sense of \cite{BusKut98}. In \cite[Section 1.2.2]{Mar23a}, we canonically associate such a representation $\tau_{a_w}$ to $\pi_{a_w}$ (called the \emph{BK-type} of $\pi_{a_w}$), assuming that $\sigma_{a_w}$ is supercuspidal, using a uniqueness result of \cite{Pas05}. In that case, there exists a unique such an embedding $\alpha_{w,j}$ (up to scalar) and \cite{Mar23a} is concerned about constructing canonical lifts of $\alpha_{a_w}$ to an embedding of $\tau_{a_w}$ into $\pi_{a_w}$. In later projects, the author plans to associated such BK-types to any $\pi_{a_w}$ without assuming that $\sigma_{a_w}$ is supercuspidal, using the theory of covers developed in \cite{BusKut98, BusKut99}. Note that similar statements can be made about $\tau_{b_w}$ and $\tau_w := \tau_{a_w} \boxtimes \tau_{b_w}$.
\end{remark}

For each $j = 1, \ldots, r_w$, fix a vector $\phi_{w,j}$ in the image of $\alpha_{w,j}$ and consider
\[
    \phi_{a_w}^0 := 
        \bigotimes \limits_{j=1}^{t_w}
            \phi_{w,j}
            \ \ \ ; \ \ \
    \phi_{b_w}^0 := 
        \bigotimes \limits_{j=t_w+1}^{r_w}
            \phi_{w,j}
\]
as vectors in the image of $\alpha_{a_w}$ and $\alpha_{b_w}$ respectively.

Let $\varphi_{a_w} \in \ind{P_{a_w}}{G_w(a_w)} \sigma_{a_w}$ be the unique function fixed by $I_{a_w, r}$ that has support $P_{a_w} I_{a_w, r}$ and
\begin{equation} \label{def varphi aw}
    \varphi_{a_w}(\gamma) = \sigma_{a_w}(\gamma) \phi_{a_w}^0 = \tau_{a_w}(\gamma) \phi_{a_w}^0 \ ,
\end{equation}
for all $\gamma \in I_{a_w, r}^0$. Denote its image in $\pi_{a_w}$ by $\phi_{a_w}$. 

\begin{remark} \label{covers of types to P-Iwahori}
    Here, we implicitly identify $\tau_{a_w}$ with its image in $\sigma_{a_w}$ and as a representation of $I_{a_w, r}^0$ that factors through $I_{a_w, r}^0/I_{a_w, r} \cong L_{a_w}(\OO_w/\p_w^r\OO_w)$. In what follows, we similarly identify $\tau_{b_w}$ (resp. $\cg{\tau}_{a_w}$, $\cg{\tau}_{b_w}$) with its cover as a representation of $\tp{I}_{b_w, r}^0$ (resp. $\tp{I}_{a_w, r}^0$, ${I}_{b_w, r}^0$) contained in $\sigma_{b_w}$ (resp. $\cg{\sigma}_{a_w}$, $\cg{\sigma}_{b_w}$).
\end{remark}

Let $\varphi_{b_w} \in \ind{P^{\opp}_{b_w}}{G_w(b_w)} \sigma_{b_w}$ be the unique function whose support is $P_{b_w}^{\opp} \tp{I}_{b_w, r}$ such that
\begin{equation} \label{def varphi bw}
    \varphi_{b_w}(\gamma) = \tau_{b_w}(\gamma) \phi_{b_w}^0 \ ,
\end{equation}
for all $\gamma \in \tp{I}_{b_w, r}^0$. Let $\phi_{b_w}$ denote its image in $\pi_{b_w}$. 

Lastly, consider the unique function $\varphi_w \in \ind{P_w}{G_w} \sigma_w$ fixed by $I_{w,r}$ whose support is $P_wI_{w,r}$ and
\begin{equation}\label{def varphi w}
    \varphi_w(\gamma) = \tau_w(\gamma) (\phi_{a_w}^0 \otimes \phi_{b_w}^0) \ ,
\end{equation} 
for all $\gamma \in I_{w,r}^0$, where $\tau_w = \tau_{a_w} \boxtimes \tau_{b_w}$. Here, we identify $\tau_w$ with its cover from $K_w$ to $I_{w,r}^0$, as in Remark \ref{covers of types to P-Iwahori}.

For our purposes, it is more convenient to work with the image of $\varphi_w$ via the map 
$
    \ind{P_w}{G_w} \sigma_w
        \to 
    \ind{P_{a_w,b_w}}{G_w} \pi_{a_w, b_w}
$ induced by the maps in \eqref{def pi aw and pi bw} and \eqref{iso induced rep}. We denote this image by $\varphi_w$ again, which should not cause any confusion since we will only ever work with $\varphi_w$ in $\ind{P_{a_w,b_w}}{G_w} \pi_{a_w, b_w}$ from now on. 

One easily checks that the support of $\varphi_w$ is $P_{a_w,b_w}I_{w,r}$ and
\[
    \varphi_w(\gamma) = \tau_w(\gamma) (\phi_{a_w} \otimes \phi_{b_w}) \ ,
\]
for all $\gamma \in I_{w,r}^0$. Let $\phi_w$ be the image of $\varphi_w$ in $\pi_w$.

\begin{remark} \label{If pi Paord}
    If $\pi$ is $P$-anti-ordinary of level $r \gg 0$ at $p$ as a representation of $G_1(\AA)$, or equivalently $\pi_w$ is $P_w$-anti-ordinary, then $\phi_{a_w}$ (resp. $\phi_{b_w}$, $\phi_w$) is a $P_{a_w}$-anti-ordinary (resp. $\tp{P}_{b_w}$-anti-ordinary, $P_w$-anti-ordinary) vector of level $r$ and type $\tau_{a_w}$ (resp. $\tau_{b_w}$, $\tau_w$) as in \cite[Theorem 4.3]{Mar23a}. The precise definitions of these notions will not play a role in the rest of this article, so we omit them. See \cite[Section 4.1]{Mar23a} for more details.
\end{remark}

The following lemma is trivial since $I_{w,r}$ is normalized by $G_w(a_w) \times G_w(b_w)$ (the product taking place in $G_w$) but we include it here since it is used several times implicitly in the computations of Section \ref{comp above p}.

\begin{lemma} \label{Iwahori fixed}
	For any $g \in G_w(a_w) \times G_w(b_w)$, $\pi_w(g)\varphi_w$ is fixed by $I_{w,r}$.
\end{lemma}

We now proceed similarly by constructing explicit vectors related to $\cg{\pi}_w$. Since $\sigma_{w,j}$ is admissible, for $j = 1, \ldots, r_w$, we also have an embedding $\cg{\alpha}_{w,j} : \cg{\tau}_{w,j} \to \cg{\sigma}_{w,j}$ of $K_{w,j}$-representations. We identify the natural contragredient pairing on $\tau_{w,j} \times \cg{\tau}_{w,j}$ with the restriction of $\brkt{\cdot}{\cdot}_{\sigma_{w,j}}$ via their fixed embedding in $\sigma_{w,j} \times \cg{\sigma}_{w,j}$. 

\begin{remark} \label{cg tau BK type cg sigma}
    If $\tau_{w,j}$ is the BK-type of $\sigma_{w,j}$ as in Remark \ref{tau BK type sigma}, then $\cg{\tau}_{w,j}$ is also the BK-type of $\cg{\sigma}_{w,j}$. In that case, such maps $\cg{\alpha}_{w,j}$ again exist and are unique up to scalar.
\end{remark}

Fix a vector $\cg{\phi}_{w,j} \in \cg{\sigma}_{w,j}$ in the image of $\cg{\alpha}_{w,j}$ such that $\brkt{\phi_{w,j}}{\cg{\phi}_{w,j}}_{\sigma_{w,j}} = 1$ and define
\[
    \cg{\phi}_{a_w}^0 := 
        \bigotimes \limits_{j=1}^{t_w}
            \cg{\phi}_{w,j}
            \ \ \ ; \ \ \
    \cg{\phi}_{b_w}^0 := 
        \bigotimes \limits_{j=t_w+1}^{r_w}
            \cg{\phi}_{w,j}
\]
as vectors in $\cg{\sigma}_{a_w}$ and $\cg{\sigma}_{b_w}$ respectively.

Assume there exists a vector $\cg{\phi}_{a_w}$ in $\cg{\pi}_{a_w}$ fixed by $\tp{I}_{a_w,r}$ such that the support of its image $\cg{\varphi}_{a_w}$ in $\ind{P_{a_w}}{G_w(a_w)} \cg{\sigma}_{a_w}$ contains $P_{a_w} \tp{I}_{a_w, r}$ and that
\begin{equation} \label{def cg varphi aw}
    \cg{\varphi}_{a_w}(\gamma) 
    = 
        \cg{\tau}_{a_w}(\gamma) \cg{\phi}_{a_w}^0 \ , \ \forall \gamma \in \tp{I}^0_{a_w,r} \ .
\end{equation}

Similarly, assume there exists a vector $\cg{\phi}_{b_w}$ in $\cg{\pi}_{b_w}$ fixed by $I_{b_w,r}$ such that the support of its image $\cg{\varphi}_{b_w}$ in $\ind{P_{b_w}}{G_w(b_w)} \cg{\sigma}_{b_w}$ contains $P_{b_w} I_{b_w, r}$ and that
\begin{equation} \label{def cg varphi bw}
    \cg{\varphi}_{b_w}(\gamma) 
    = 
        \cg{\tau}_{b_w}(\gamma) \cg{\phi}_{b_w}^0 \ , \ \forall \gamma \in I^0_{b_w,r} \ .
\end{equation}

Lastly, assume there exists a vector $\cg{\phi}_w$ in $\cg{\pi}_w$ fixed by $\tp{I}_{w,r}$ such that the support of its image $\cg{\varphi}_w$ in $\ind{P_{a_w,b_w}}{G_w} \cg{\pi}_{a_w, b_w}$ contains $P_w \tp{I}_{w, r}$ and that
\begin{equation} \label{def cg varphi w}
    \cg{\varphi}_w(\gamma) 
    = 
        \cg{\tau}_w(\gamma) (\cg{\phi}_{a_w} \otimes \cg{\phi}_{b_w}) \ , \ \forall \gamma \in \tp{I}^0_{w,r} \ .
\end{equation}

\begin{remark} \label{If cg pi Paord}
    As in Remark \ref{If pi Paord}, assume that $\pi$ is $P$-anti-ordinary of level $r \gg 0$ at $p$ as a representation of $G_1(\AA)$, omitting the precise definition of this notions. Most importantly, if this holds $\cg{\pi}_w$ is a $P_w$-anti-ordinary of level $r \gg 0$ (as a local factor of an automorphic representation of $G_2(\AA)$) and \cite[Lemma 4.6 (ii)]{Mar23a}, proves the existence and uniqueness of vectors $\cg{\phi}_{a_w}$, $\cg{\phi}_{b_w}$ and $\cg{\phi}_w$. In the last case, one needs to use the isomorphism \eqref{iso induced rep} to compare \emph{loc. cit.} with our notation here.

    In particular, in that case $\cg{\phi}_{a_w}$ (resp. $\cg{\phi}_{b_w}$, $\cg{\phi}_w$) is $\tp{P}_{a_w}$-anti-ordinary (resp. $P_{b_w}$-anti-ordinary, $\tp{P}_w$-anti-ordinary) of type $\cg{\tau}_{a_w}$ (resp. $\cg{\tau}_{b_w}$, $\cg{\tau}_w$) in the sense of \cite[Sections 3.2 and 4.1]{Mar23a}.
\end{remark}

\subsubsection{Inner products between test vectors.}
Observe that the intersection of the support of $\varphi_w$ with $\GL_n(\OO_w)$ is $P_{a_w,b_w}I_{w,r} \cap \GL_n(\OO_w) = I^0_{a_w,b_w, r}$. Therefore,
\begin{align*}
    \la \phi_w, \cg{\phi}_w \ra_{\pi_w} 
    = 
    \int_{I^0_{a_w, b_w, r}} 
        (
            \varphi_w(k),
            \cg{\varphi}_w(k)
        )_{a_w,b_w} 
    d^\times k \ ,
\end{align*}

Write any $k \in I^0_{a_w, b_w, r}$ as
\[
    k = 
    \begin{pmatrix}
	1 & B \\
	0 & 1
    \end{pmatrix}
    \begin{pmatrix}
	A & 0 \\
	0 & D
    \end{pmatrix}
    \begin{pmatrix}
	1 & 0 \\
	C & 1
    \end{pmatrix} 
\]
where $A \in \GL_{a_w}(\OO_w)$, $D \in \GL_{b_w}(\OO_w)$, $B \in M_{a_w \times b_w}(\OO_w)$ and $C \in \p_w^rM_{b_w \times a_w}(\OO_w)$. Since
$
    \begin{pmatrix}
        1 & B \\ 0 & 1
    \end{pmatrix}
$ 
is in $P_{a_w, b_w}$ and 
$
    \begin{pmatrix}
        1 & 0 \\ C & 1
    \end{pmatrix}
$
is in both $I_{w,r}$ and $\tp{I}_{w,r}$, we see that
\[
    \varphi_w(k) = 
        \varphi_w
        \left( 
		\begin{pmatrix}
			A & 0 \\
			0 & D
		\end{pmatrix} 
	\right)
	= 
        \pi_{a_w}(A) \phi_{a_w} 
        \otimes \pi_{b_w}(D) \phi_{b_w}
\]
and
\[
    \cg{\varphi}_w(k) = 
	\cg{\varphi}_w
        \left(
		\begin{pmatrix}
			A & 0 \\
			0 & D
		\end{pmatrix}
	\right)
	= 
        \cg{\pi}_{a_w}(A) \cg{\phi}_{a_w} 
        \otimes \cg{\pi}_{b_w}(D) \cg{\phi}_{b_w} \ ,
\]
so we obtain
\begin{equation} \label{size varphi}
    \la 
        \phi_w, 
        \cg{\phi}_w 
    \ra_{\pi_w} 
        = 
    \vol(I^0_{a_w,b_w,r}) 
    \la 
            \phi_{a_w}, 
            \cg{\phi}_{a_w}
        \ra_{\pi_{a_w}} 
    \la 
            \phi_{b_w}, 
            \cg{\phi}_{b_w}
        \ra_{\pi_{b_w}} \ .
\end{equation}

Similar arguments yield
\begin{equation} \label{size phia}
    \la 
        \phi_{a_w}, 
        \cg{\phi}_{a_w} 
    \ra_{\pi_{a_w}} = 
	\vol(I^0_{P_{a_w}, r})
	( 
            \phi^0_{a_w}, 
            \cg{\phi}^0_{a_w}
        )_{{a_w}}= 
	\vol(I^0_{P_{a_w}, r})
\end{equation}
and
\begin{equation} \label{size phib}
    \la 
        \phi_{b_w}, 
        \cg{\phi}_{b_w} 
    \ra_{\pi_{b_w}} = 
	\vol(I^0_{P_{b_w}, r})
	( 
            \phi^0_{b_w},
            \cg{\phi}^0_{b_w} 
        )_{{b_w}} = 
	\vol(I^0_{P_{b_w}, r}) \ ,
\end{equation}
using the fact that $(\phi^0_{w, j}, \cg{\phi}^0_{w, j}) = 1$ for each $1 \leq j \leq r_w$. Ultimately, we obtain
\begin{equation} \label{size phi}
    \la 
        \phi_w, 
        \cg{\phi}_w 
    \ra_{\pi_w} 
    = 
        \vol(I^0_{a_w,b_w,r}) 
        \vol(I^0_{P_{a_w}, r})
        \vol(I^0_{P_{b_w}, r})
    =
        \vol(I^0_{w,r}) \ ,
\end{equation}
which in particular is nonzero.

%% file: S3/3-1.tex
\section{Local Siegel-Weil sections associated to types.} 
\label{doubling method}
For $i=1,2$, let $\PP_i = (\KK, c, \OO, L_i, \la \cdot, \cdot \ra_i, h_i)$, $L_i \otimes \ZZ_p = L_i^+ \oplus L_i^-$ be the PEL datum described in Section \ref{def G1 and G2}. Furthermore, for $i=3,4$, define similar PEL datum $\PP_i = (\KK, c, \OO, L_i, \la \cdot, \cdot \ra_i, h_i)$, $L_i \otimes \ZZ_p = L_i^+ \oplus L_i^-$, again of unitary type in the sense of \cite[Section 2.2]{EHLS}, where
\begin{align*}
    \PP_3 &:= (\KK \times \KK, c \times c, \OO \times \OO, L_1 \oplus L_2, \la \cdot, \cdot \ra_1 \oplus \la \cdot, \cdot \ra_2, h_1 \oplus h_2), L_3^\pm := L_1^\pm \oplus L_2^\pm \\
    \PP_4 &:= (\KK, c, \OO, L_3, \la \cdot, \cdot \ra_3, h_3), L_4^\pm := L_3^\pm
\end{align*}

Let $G_i$ denote the similitude unitary group in \eqref{def sim uni gp} associated to $\PP_i$. Similarly, let $\nu_i$ be the similitude character of $G_i$ and $U_i = \ker \nu_i$. Using this notation, we recall the doubling method of \cite{Gar84, GPSR87} on unitary groups, following standard approach described in \cite[Section 4.1]{EHLS}.

\subsection{Siegel Eisenstein series and the doubling method.}
Given any number field $F/\QQ$, we write $\absv{\cdot}_F$ for the standard absolute value on $\AA^\times_F$ (instead of $\absv{\cdot}_{\AA_F}$). Moreover, in this section, we write $\AA$ for $\AA_{\QQ}$ and we write $G$ for $G_4$.

\subsubsection{Siegel parabolic.}
Let $W = V \oplus V$, equipped with $\brkt{\cdot}{\cdot}_W := \brkt{\cdot}{\cdot}_V \oplus (- \brkt{\cdot}{\cdot}_V)$, be the Hermitian vector space associated to $G_4$. Consider the subspaces $V^d = \{(x,x) \in W : x \in V\}$ and $V_d = \{(x,-x) \in W : x \in V\}$. We identify each of them with $V$ via projection on the first factor. The direct sum $W = V_d \oplus V^d$ is a polarization of $\brkt{\cdot}{\cdot}_W$.

Let $P_\sgl \subset G$ denote the stabilizer of $V^d$ under the right-action of $G$, a maximal $\QQ$-parabolic subgroup. Let $M \subset P_\sgl$ denote the Levi subgroup that also stabilizes $V_d$. The unipotent radical of $P_\sgl$ is the subgroup $N$ that fixes both $V^d$ and $W/V^d$ and clearly, $P_{\sgl}/N \cong M$. Furthermore, there is a canonical identification $M \xrightarrow{\sim} \GL_{\KK}(V) \times \Gm$ via $m \mapsto (\Delta(m), \nu(m))$, where $\Delta$ is the projection
\[
     P_\sgl \to \GL_\KK(V^d) = \GL_\KK(V) \ ,
\]
whose inverse is given by 
$
    (A, \lambda) 
        \mapsto 
    \diag(\lambda (A^*)^{-1}, A)
$, where $A^* = \tp{A}^c$. 

\subsubsection{Induced Representations.} \label{siegel induced rep}
Let $\chi : \KK^\times \backslash \AA_{\KK}^\times \to \CC^\times$ be a unitary Hecke character. It factors as $\chi = \bigotimes \limits_{w} \chi_w$, where $w$ runs over all places of $\KK$.

For convenience, define the character $\nabla$ of $P_\sgl(\AA)$ as
\[
    \nabla(-) = 
    \absv{ 
        \nm_{\KK/\KK^+} \circ 
        \det \circ 
        \Delta(-)
    }_{\KK^+} \cdot
    \absv{
        \nu(-)
    }_{\KK^+}^{-n} = 
    \absv{ 
        \det \circ 
        \Delta(-)
    }_\KK \cdot
    \absv{
        \nu(-)
    }_\KK^{-n/2} \ ,
\]
where $\nm_{\KK/E}$ is the usual norm homomorphism $\AA_{\KK} \to \AA_E$. One readily checks that $G_1(\AA)$, via its natural diagonal inclusion in $G_4(\AA)$, is in the kernel of $\nabla$. Moreover, the modulus character $\delta_{\sgl}$ of $P_\sgl(\AA_\QQ)$ equals $\nabla^{n}$.

Let $s \in \CC$, and define the smooth and normalized induction
\begin{align*}
    I(\chi, s) 
    &= 
    \ind{P_\sgl(\AA)}{G(\AA)}
        \left(
            \chi
            \left(
                \det \circ \Delta(-)
            \right)
                \cdot
            \nabla(-)^{-s}
        \right) \ .
\end{align*}

This degenerate principal series is identical to the one in \cite[Section 4.1.2]{EHLS}. It is also equal to the smooth, unnormalized parabolic induction
\[
    I(\chi, s) =
    \Ind_{P_\sgl(\AA)}^{G(\AA)}
        \left(
            \chi
            \left(
                \det \circ \Delta(-)
            \right) 
                \cdot
            \nabla(-)^{\frac{n}{2}-s}
        \right)
\]
and factors as a restricted tensor product of local induced representations
\[
    I(\chi, s) = \bigotimes_{v} I_v(\chi_v, s) \ ,
\]
where $v$ runs over all places of $\QQ$ and $\chi_v = \bigotimes \limits_{w | v} \chi_w$.

\subsubsection{Siegel Eisenstein series.} \label{Sgl Eis series}
Given a section $f = f_{\chi, s}$ of $I(\chi, s)$, one constructs the \emph{standard (nonnormalized)} Eisenstein series
\begin{equation} \label{def Eis series}
    E_f(g) = 
        \sum_{
            \gamma \in P(\QQ) \backslash G(\QQ)
        } 
        f(\gamma g)
\end{equation}
as a function on $G(\AA)$. It converges on the half-plane Re$(s) > n/2$ and if $f$ is right-$K$-finite, for some maximal compact open subgroup $K \subset G$, it admits a meromorphic continuation on $\CC$.

\subsubsection{Zeta integrals.} \label{zeta integrals}
Let $f = f_{\chi, s} \in I(\chi, s)$. Fix any $\varphi \in \pi$ and $\cg{\varphi} \in \cg{\pi}$ and consider the corresponding Rankin-Selberg zeta integral
\[
    I(\varphi, \cg{\varphi}, f; \chi, s) := 
    \int_{
        Z_3(\AA)G_3(\QQ) \backslash G_3(\AA)
    } 
        E_f(g_1, g_2)
        \varphi(g_1)
        \cg{\varphi}(g_2)
        \chi^{-1}(\det g_2)
    d(g_1,g_2)
\]

The arguments of \cite{GPSR87} show that for Re$(s)$ large enough,
\[
    I(\varphi, \cg{\varphi}, f; \chi, s) 
        = 
    \int_{U_1(\AA)} 
        f_{\chi, s}(u, 1)
        \brkt{
            \pi(u) \varphi
        }{
            \cg{\varphi}
        }_\pi
    du 
\]

Assume that $f = \otimes_v f_v$, as the product runs over all places $v$ of $\QQ$, for some local Siegel-Weil sections $f_v = f_{\chi, s, v} \in I_v(\chi_v, s)$. Furthermore, assume $\varphi$ and $\cg{\varphi}$ are ``pure tensors'', i.e. $\varphi = \otimes_v \varphi_v$ and $\cg{\varphi} = \otimes_v \cg{\varphi}_v$ according to the factorization \eqref{facto pi}. 

Then
\[  
    I(\varphi, \cg{\varphi}, f; \chi, s) 
        = 
    \prod_v
        I_v(\varphi_v, \cg{\varphi}_v, f_v; \chi_v, s) \cdot  
        \brkt{\varphi}{\cg{\varphi}}
\]
where
\begin{equation} \label{def local zeta integrals}
    I_v(\varphi_v, \cg{\varphi}_v, f_v; \chi_v, s) 
        =
    \frac{
        \int_{U_{1,v}} 
            f_{\chi, s, v}(u, 1)
            \brkt{
                \pi_v(u) \varphi_v
            }{
                \cg{\varphi}_v
            }_{\pi_v}
        du
    }{
        \brkt{\varphi_v}{\cg{\varphi}_v}_{\pi_v}
    }
\end{equation}
for any place $v$ of $\QQ$. Let $Z_v$ denote the numerator of the fraction above. In Section \ref{comp above p}, we compute the $p$-adic zeta integral $Z_p$ associated to a specific choice of Siegel-Weil section constructed in the next section and specific test vectors $\varphi_p$, $\cg{\varphi}_p$.

%% file: S3/3-2.tex
\subsection{Choice of Siegel-Weil sections at $p$} \label{SW section at p}
For each places $w \in \Sigma_p$ of $\KK$, fix an isomorphism $\KK_w = \KK_{\overline{w}}$. Then, the identification \eqref{prod G over Zp} for $G_4$ induces an identification of $P_\sgl(\QQ_p)$ with $\QQ_p^\times \times \prod_{w \in \Sigma_p} P_n(\KK_w)$, where $P_n \subset \GL_{\KK}(W)$ is the parabolic subgroup stabilizing $V^d$.

Let $\chi_p = \otimes_{w \mid p} \chi_w$ and, given $s \in \CC$, view $\chi_p \cdot \absv{-}_p^{-s}$ as a character of $P_\sgl(\QQ_p)$. One readily checks that the restriction to $\prod_{w \in \Sigma_p} P_n(\KK_w)$ corresponds to the product over $w \in \Sigma_p$ of the characters $\psi_{w,s} : P_n(\KK_w) \to \CC^\times$ defined as
\[
    \psi_{w,s}
    \left( 
        \begin{pmatrix}
            A & B \\
            0 & D
        \end{pmatrix}
    \right)
        =
    \chi_w(\det D) \chi_{\ol{w}}(\det A^{-1}) \cdot \absv{\det A^{-1}D }_w^{-s} \ ,
\]
by writing element of $P_n$ according to the direct sum decomposition $W = V_d \oplus V^d$.

Let $W_w = W \otimes_{\KK} \KK_w$ and choose any $f_{w,s} \in \ind{P_n(\KK_w)}{\GL_{\KK_w}(W_w)} \psi_{w,s}$, for each $w \in \Sigma$. Then, it is clear that the section
\begin{equation} \label{fp and fws}
    f_p(g) = f_{p, \chi, s}(g) := 
        \absv{\nu}_p^{(s - \frac{n}{2})\frac{n}{2}} 
        \prod_{w \in \Sigma_p} f_{w,s}(g_w) \ , \ \ \ g = (\nu, (g_w)_w) \in G(\QQ_p)
\end{equation}
is in $I_p(\chi_p, s)$.

\begin{remark}
    The strategy below is to construct such $f_{w,s}$ (and hence $f_p$) from specific Schwartz functions. This approach is already used in \cite[Section 2.2.8]{Eis15} and \cite[Section 4.3.1]{EHLS}. In fact, our argument owes a great deal to their work and the details they carefully provide. 
    
    The novelty here is that we associate Schwartz functions to finite dimensional representations (implicitly viewed as Bushnell-Kutzko types), instead of characters. This generalization is necessary when working with an automorphic representation $\pi$, as in Section \ref{conv test vectors}, with non-trivial supercuspidal support. This is the case when $\pi$ is $P$-ordinary (or $P$-anti-ordinary) but not ordinary (or anti-ordinary).
\end{remark}

\subsubsection{Locally Constant Matrix Coefficients.}
Let the characters $\chi_{w,1}$ and $\chi_{w,2}$ of Section \ref{conv test vectors} be $\chi_w$ and $\chi_{\overline{w}}^{-1}$ respectively. For the remainder of this section, we use the notation of Section \ref{conv test vectors}, where in particular $G = G_1$, freely.

Let $\mu'_{w,j} : K_{w,j} \to \CC$ be the matrix coefficient defined as
\[
    \mu'_{w,j}(X) = 
    \begin{cases}
        \langle 
            \phi_{w,j},
            \cg{\tau}_{w,j}(X)
            \cg{\phi}_{w,j} 
        \rangle_{\sigma_{w,j}}, & 
        \text{if $j=1, \ldots, t_w$}, \\
        \langle 
            \tau_{w,j}(X)
            \phi_{w,j},
            \cg{\phi}_{w,j} 
        \rangle_{\sigma_{w,j}}, & 
        \text{if $j=t_w+1, \ldots, r_w$}.
    \end{cases}
\]

The products $\mu'_{a_w} = \bigotimes_{j=1}^{t_w} \mu'_{w,j}$ and $\mu'_{b_w} = \bigotimes_{j=t_w+1}^{r_w} \mu'_{w,j}$ on $K_{a_w}$ and $K_{b_w}$, respectively, are the matrix coefficients
\[
    \mu'_{a_w}(X) = 
    (
        \phi_{a_w}^0, 
        \cg{\tau}_{a_w}(X) 
        \cg{\phi}_{a_w}^0 
    )_{a_w}
        \ \ \ ; \ \ \
    \mu'_{b_w}(X) = 
    ( 
        \tau_{b_w}(X)
        \phi_{b_w}^0, 
        \cg{\phi}_{b_w}^0
    )_{b_w}
\]
of $\cg{\tau}_{a_w}$ and $\tau_{b_w}$ respectively. 

We now consider $\mu'_{a_w}$ as a locally constant function on $M_{a_w}(\KK_w)$ supported on $\X_w^{(1)} := \tp{I}_{a_w, r}^0 I_{a_w, r}^0$. More precisely, one readily verifies that given $X \in \X_w^{(1)}$ and any $\tp{\gamma}_1, \gamma_2 \in I_{a_w, r}^0$ such that $X = \gamma_1\gamma_2$, then
\begin{equation} \label{def extn mu aw}
    \mu'_{a_w}(X) := 
    ( 
        \tau_{a_w}(\gamma_1^{-1})
        \phi_{a_w}^0,
        \cg{\tau}_{a_w}(\gamma_2)
        \cg{\phi}_{a_w}^0 
    )_{a_w}
\end{equation}
is well-defined. Similarly, we extend $\mu'_{b_w}$ to a locally constant function on $M_{b_w}(\KK_w)$ supported on $\X_w^{(4)} := \tp{I}_{b_w, r}^0 I_{b_w, r}^0$ via
\begin{equation} \label{def extn mu bw}
    \mu'_{b_w}(X) := 
    ( 
        \tau_{b_w}(\gamma_2)
        \phi_{b_w}^0,
        \cg{\tau}_{b_w}(\gamma_1^{-1})
        \cg{\phi}_{b_w}^0 
    )_{b_w},
\end{equation}
where $X \in \X_w^{(4)}$ and $\tp{\gamma}_1, \gamma_2 \in I_{b_w, r}^0$ are any elements such that $X = \gamma_1\gamma_2$.
\begin{remark}
    Note here that we are extending $\tau_{a_w}$ (resp. $\tau_{b_w}$, $\cg{\tau}_{a_w}$, $\cg{\tau}_{b_w}$) to a representation of $\tp{I}_{a_w, r}^0$ (resp. $\tp{I}_{b_w, r}^0$, $I_{a_w, r}^0$, $I_{b_w, r}^0$). Although this is the opposite of what is done in Remark \ref{covers of types to P-Iwahori}, this is exactly the necessary step to obtain the correct cancellation at the end of the proof of Theorem \ref{main thm - comp of Zw}.
\end{remark}

Let $\mu_{a_w}(A) := \chi_{2,w}^{-1}\mu'_{a_w}$ and $\mu_{b_w} := \chi_{1,w}\mu'_{b_w}$. Let $\X_w \subset M_n(\OO_w)$ be the set of matrices 
$
    \begin{pmatrix}
        A & B \\
        C & D
    \end{pmatrix}
$
such that $A \in \X_w^{(1)}$, $B \in M_{a_w \times b_w}(\OO_w)$, $C \in M_{b_w \times a_w}(\OO_w)$ and $D \in \X_w^{(4)}$. 

We define a locally constant function $\mu_w$ on $M_{n}(\KK_w)$ supported on $\X$ via
\begin{equation} \label{def ext mu w}
    \mu_w\left(
        \begin{pmatrix}
            A & B \\
            C & D
        \end{pmatrix}
    \right)
    = \mu_{a_w}(A)\mu_{b_w}(D) \ ,
\end{equation}
for all 
$
    \begin{pmatrix}
        A & B \\
        C & D
    \end{pmatrix}
    \in \X_w
$.

Observe that the set $\X_w$ contains the subgroup $\GG_w = \GG_w(r) \subset \GL_n(\OO_w)$ consisting of matrices whose terms below the $(n_{w,j} \times n_{w,j})$-blocks along the diagonal are in $p^r$ and such that the upper right $(a_w \times b_w)$-block is also in $p^r$. 

In particular, there is an obvious decomposition $\GG_w = \GG_l (I^0_{a_w,r} \times I^0_{b_w,r}) \GG_u$. By abuse of notation, given $B \in M_{a_w \times b_w}(\KK_w)$ or $C \in M_{b_w \times a_w}(\KK_w)$, we sometimes write $B \in \mathfrak{G}_u$ or $C \in \mathfrak{G}_l$ when we mean
\[
	\begin{pmatrix}
		1 & B \\
		0 & 1
	\end{pmatrix} \in \mathfrak{G}_u \ \ \ \ \ \text{or} \ \ \ \ \ 
	\begin{pmatrix}
		1 & 0 \\
		C & 1
	\end{pmatrix} \in \mathfrak{G}_l \ .
\]

\subsubsection{Choice of Schwartz functions.} \label{choice Schw func}
Let $\Phi_{1,w}: M_n(\KK_w) \to \CC$ be the locally constant function supported on $\GG_w$ such that 
\[
    \Phi_{1,w}(X) = 
    \mu_w(X)
\]
for all $X \in \GG_w$. Furthermore, define the locally constant functions
\begin{align*}
    \nu_{\bullet}(z) = 
        \chi_{w,1}^{-1} \chi_{w,2} 
        \mu_{\bullet}(z)
    \ \ \ ; \ \ \
    \phi_{\nu_{\bullet}}(z) = 
        \nu_{\bullet}(-z) \ ,
\end{align*}
where $\bullet$ denotes $a_w$, $b_w$ or $w$ and $z$ is in the appropriate domain.

Let $\Phi_{2,w} : M_n(\QQ_p) \to \CC$ be 
\[
	\Phi_{2,w}(x) = (\nu_w)^\wedge(x) = \int_{M_n(\KK_w)} \phi_{\nu_w}(y) e_w(\tr(yx)) dy
\]

\begin{remark}
    The definition of $\mu_w$ and its twist $\nu_w$ on $\X_w$ allows us to generalize the function denoted $\phi_{\nu_v}$ in \cite[Section 4.3.1]{EHLS}. In \emph{loc. cit.}, the BK-types are all characters, in which case $\tau_\bullet$ is equal to $\mu'_\bullet$. The ``telescoping product'' in the definition of $\phi_{\nu_v}$ is simply a formula that expresses the extension of these characters to $I_{\bullet}^0$ and $\tp{I}_{\bullet}^0$ simultaneously. Our alternative is to use extensions of (matrix coefficients of) BK-types such as in Equations \eqref{def extn mu aw} and \eqref{def ext mu w}.
\end{remark}

\begin{remark}
	This Fourier transform in the definition of $\Phi_{2,w}$ is slightly different than the one in \cite[Section 2.2.8]{Eis15} and \cite[Section 4.3.1]{EHLS}. It is the same as the one involved in the Godement-Jacquet functional equation \cite{Jac79}.
\end{remark}

\begin{lemma} \label{Phi2 support}
    Given 
    $X = 
        \begin{pmatrix} 
            A & B \\ 
            C & D 
        \end{pmatrix}$ 
    with $A \in M_{a_w \times a_w}(\KK_w)$, $B, \tp{C} \in M_{a_w \times b_w}(\KK_w)$ and $D \in M_{b_w \times b_w}(\KK_w)$, one can write
    \[
	\Phi_{2,w}(X) = 
            \Phi_w^{(1)}(A)
            \Phi_w^{(2)}(B)
            \Phi_w^{(3)}(C)
            \Phi_w^{(4)}(D)
    \]
    with
    \begin{align*}
	\Phi_w^{(2)} = 
            \Char_{M_{a_w \times b_w}(\OO_w)},
            &\hspace*{1cm} 
        \Phi_w^{(3)} = 
            \Char_{M_{b_w \times a_w}(\OO_w)}, \\
        \supp(\Phi_w^{(1)}) \subset 
            \p_w^{-r}M_{a_w \times a_w}(\OO_w), 
            &\hspace*{1cm} 
        \supp(\Phi_w^{(4)}) \subset 
            \p_w^{-r}M_{b_w \times b_w}(\OO_w),
    \end{align*}
    where $r$ is as in Inequality \eqref{large enough r}.
\end{lemma}

\begin{proof}
    The definitions of $\phi_{\nu_{a_w}}$, $\phi_{\nu_{b_w}}$ and $\phi_{\nu_w}$ immediately imply
    \begin{align*}
        \Phi_{2,w}(X) 
        &= 
        \int_{\X_w} 
            \phi_{\nu_w}
            \left(
                \begin{pmatrix}
			     \alpha & \beta \\
			     \gamma & \delta
            \end{pmatrix}
            \right)
            e_w(
                \tr(
                    \alpha A + \beta B + \gamma C + \delta D
                )
            )
        d\alpha d\beta d\gamma d\delta \\
        &= 
        \int_{\X_w^{(1)}}
		\phi_{\nu_{a_w}}
            \left(
                \alpha
            \right)
            e_w(\tr(\alpha A))
        d\alpha 
        \int_{\X_w^{(4)}}
            \phi_{\nu_{b_w}}
            (
                \delta
            )
            e_w(\tr(\delta D))
        d\delta \\
        &\times 
        \Char_{M_{a_w \times b_w}(\OO_w)}(B) \Char_{M_{b_w \times a_w}(\OO_w)}(C) \\ 
    \end{align*}
	
    Then, we may conclude as in the proof of \cite[Lemma 4.3.2 (ii)]{EHLS} by observing
    \begin{align*}
	&\Phi_w^{(1)}(A) 
        := 
        \int_{\X_w^{(1)}}
		\phi_{\nu_{a_w}}
            \left(
                \alpha
            \right)
		e_w(\tr(\alpha A))
	d\alpha \\
	&= 
        \vol(\p_w^rM_{a_w}(\OO_w))
	\sum_{
            \alpha \in \X_w^{(1)} \text{ mod } \p_w^r
        }
    	\phi_{\nu_{a_w}}
                \left(
                    \alpha
                \right)
    	e_w(\tr{\alpha A}) 
    	\Char_{\p_w^{-r}M_{a_w}(\OO_w)}(A) \ ,
    \end{align*}
    and
    \begin{align*}
	&\Phi_w^{(4)}(D) 
        := 
        \int_{\X_w^{(4)}}
		\phi_{\nu_{b_w}}
            \left(
                \delta
            \right)
		e_w(\tr(\delta D))
	d\delta \\
	&= 
        \vol(\p_w^r M_{b_w}(\OO_w))
	\sum_{
            \delta \in \X_w^{(4)} \text{ mod } p^r
        }
    	\phi_{\nu_{b_w}}
            \left(
                \delta
            \right)
    	e_w(\tr{\delta D}) 
	   \Char_{\p_w^{-r}M_{b_w}(\OO_w)}(D) \ .
    \end{align*}
\end{proof}

Define the Schwartz function $\Phi_w : M_{n \times 2n}(\QQ_p) \to \CC$ as
\begin{equation} \label{def Schwartz Phi w}
	\Phi_w(X) = \Phi_w(X_1, X_2) = \frac{(\dim \tau_w)^2}{\vol(\GG_w)} \Phi_{1,w}(-X_1) \Phi_{2,w}(X_2) \ .
\end{equation}

\subsubsection{Construction of $f_{w,s}$}
For each $w \in \Sigma_p$, write $V_w = V \otimes_{\KK} \KK_w$ and use similar notation for $V_{d, w}$ and $V^d_w$. Consider the decomposition
\[
    \hom_{\KK_w}(V_w, W_w) = \hom_{\KK_w}(V_w, V_{w,d}) \oplus \hom_{\KK_w}(V_w, V_w^d), \ \ \ X = (X_1, X_2) 
\]
and its subspace
\[
    \mathbf{X} 
        := 
    \{
        X \in \hom_{\KK_w}(V_w, W_w) \mid X(V_w) = V_w^d
    \}
        =
    \{
        (0,X) \mid X : V_w \xrightarrow{\sim} V_w^d
    \}.
\]

In fact, any $X \in \mathbf{X}$ can be viewed as an automorphism of $V_w$ (by composing with the identification of $V^d$ with $V$) and hence, we identify $\mathbf{X}$ with $\GL_{\KK_w}(V_w)$. Let $d^\times X$ be the Haar measure on the latter.

Furthermore, recall that we fixed an $\OO_w$-basis of $L_{1,w}$ in Section \ref{G(Zp) to GLn}. This provides a $\KK_w$-basis of $V_w$ and, via their identification to $V$, a $\KK_w$-basis of $V_{d, w}$ and of $V^d_w$. Hence, it also induces a $\KK_w$-basis of $W_w = V_{d, w} \oplus V^d_w$.

It identifies $\isom(V_w^d, V_w)$ with $\isom(V_{w,d}, V_w)$, $\GL_{\KK_w}(V_w)$ with $\GL_n(\KK_w)$, $\GL_{\KK_w}(W_w)$ with $\GL_{2n}(\KK_w)$, $P_n(\KK_w)$ with the subgroup of $\GL_{2n}(\KK_w)$ consisting of upper-triangular $n \times n$-block matrices, and $\hom_{\KK_w}(V_w, W_w)$ with $M_{n \times 2n}(\QQ_p)$

Therefore, we now view the Schwartz function $\Phi_w : M_{n \times 2n}(\QQ_p) \to \CC$ constructed above as a function on $\hom_{\KK_w}(V_w, W_w)$. We define the section $f_{w,s} = f_{w,s}^{\Phi_w} = f^{\Phi_w}$ of $\ind{P_n(\KK_w)}{\GL_{2n}(\KK_w)} \psi_{w,s}$ by
\begin{equation} \label{def f phi w}
    f^{\Phi_w}(g) = 
        \chi_{2,w}(\det g)
        \absv{\det g}_w^{\frac{n}{2} + s}
        \int_{\mathbf{X}}
            \Phi_w(Xg)
            \chi_{w,1}^{-1}
            \chi_{w,2}(\det X)
            |\det X|_w^{n+2s}
        d^\times X
    \ ,    
\end{equation}
as in \cite[Equation (55)]{EHLS}.

\subsubsection{Construction of $f_{w,s}^+$} \label{construction f+}
Let $f_p$ be the corresponding local Siegel-Weil section at $p$, as in Equation \eqref{fp and fws}. Ahead of our computations in the next section, we write down an explicit expression for $f_p(u,1)$ for any $u \in U_1(\QQ_p)$. 

Firstly, the restriction of the isomorphsim \eqref{prod G over Zp} to $U_1$ yields an identification $U_1(\QQ_p) = \prod_{w \in \Sigma_p} U_{1,w}$, where $U_{1,w} = \GL_{\KK_w}(V_w) = \GL_n(\KK_w)$. We write $u = (u_w)_{w \in \Sigma_p}$ accordingly. One can then consider the identity \eqref{def f phi w} for $g = (u, 1)$, where the latter is with respect to the embedding $G_1 \times G_2 \hookrightarrow G_3$. To simplify the expression, it is therefore more convenient to replace the decomposition $W_w = V_{w, d} \oplus V_w^d$ with $W_w = V_w \oplus V_w$. In that case, an element $X \in \GL_n(\KK_w) = \GL_{\KK_w}(V_w)$ corresponds to an element $(X,X)$ in $\mathbf{X}$ instead of $(0,X)$.

Secondly, using the decomposition $W_w = V_w \oplus V_w$ again and the corresponding identification $W_w = V_w \oplus V_w$, consider the element
\[
    S_w = 
    \begin{pmatrix}
        1_{a_w} & 0 & 0 & 0 \\
        0 & 0 & 0 & 1_{b_w} \\
        0 & 0 & 1_{a_w} & 0 \\
        0 & 1_{b_w} & 0 & 0
    \end{pmatrix} .
\]

\begin{remark}
    As explained in \cite[Section 2.1.11]{HLS} and \cite[Remark 3.1.4]{EHLS}, the natural inclusion of Shimura varieties associated to $G_3$ and $G_4$ does not induce the natural inclusion on Igusa tower. In fact, one needs to twist the former by the matrix $S_w$ to induce the latter. This point of view will not play a role for us but we include this remark as motivation for introducing $S_w$.
\end{remark}

Lastly, replace each $f_{w,s}$ by its translation $f_{w,s}^+$ via $g \mapsto gS_w$, and let $f_p^+$ be the corresponding local Siegel-Weil section at $p$ defined by Equation \eqref{fp and fws}. In that case, for $g=(u,1)$, we obtain that $f^+_p(u,1)$ is equal to a product over $w \in \Sigma_p$ of
\[
    \chi_{2,w}(\det u_w)
    \absv{\det u_w}_w^{\frac{n}{2} + s}
    \int_{\GL_n(\KK_w)}
        \Phi_w((Xu_w,X) S_w)
        \chi_{w,1}^{-1}
        \chi_{w,2}(\det X)
        |\det X|_w^{n+2s}  
    d^\times X
\]
and we denote the above expression by $f_{w,s}^+(u_w,1) = f_w^+(u_w,1)$, as a function of $u_w \in \GL_n(\KK_w)$.

%% file: S4/4-1.tex
\section{Main calculations.} \label{comp above p}
\subsection{Local integrals at places above $p$} \label{comp of Zw}
We keep the same notation as in Sections \ref{conv test vectors} and \ref{SW section at p}. In particular, for each $w \in \Sigma_p$, consider the vector $\phi_w \in \pi_w$ and $\cg{\phi}_w \in \cg{\pi}_w$ defined in Section \ref{comp test vectors}. We now compute the $p$-adic local zeta integral $Z_p$ defined in \eqref{def local zeta integrals} for the local section $f_p^+$ constructed above and the test vectors
\begin{equation} \label{def test vec varphi p and cg varphi p}
    \varphi_p = 1 \otimes \left( \bigotimes_{w \in \Sigma_p} \phi_w \right)
        \ \ \ ; \ \ \ 
    \cg{\varphi}_p = 1 \otimes \left( \bigotimes_{w \in \Sigma_p} \cg{\phi}_w \right)
\end{equation}

In particular, $Z_p$ is equal to the product over $w \in \Sigma_p$ of
\begin{align*}
    Z_w :=
        &\int_{\GL_n(\KK_w)} 
            f_{w,s}^+(g,1) \la \pi_w(g) \phi_w, \cg{\phi}_w \ra_{\pi_w}
        d^\times g \\
    =
        &\int_{\GL_n(\KK_w)} 
            \chi_{2,w}(\det g)
            \absv{\det g}_w^{\frac{n}{2} + s}
            \int_{\GL_n(\KK_w)}
                \Phi_w((Xg, X) S_w) \\
        &\times
                \chi_{w,1}^{-1}
                \chi_{w,2}(\det X)
                |\det X|_w^{n+2s}  
            \la \pi_w(g) \phi_w, \cg{\phi}_w \ra_{\pi_w}
            d^\times X 
        d^\times g
\end{align*}

According to the decomposition $M_{n \times n}(\KK_w) = M_{n \times a_w}(\KK_w) \times M_{n \times b_w}$, write $Z_1 := Xg = [Z_1', Z_1'']$ and $Z_2 := X = [Z_2', Z_2'']$, where $Z_1'$ and $Z_2'$ (resp. $Z_1''$ and $Z_2''$) are $n \times a$-matrices (resp. $n \times b$-matrices). Then,
\[
    (Xg, X) S_w = ([Z_1', Z_2''], [Z_2', Z_1''])
\]
and 
\[
    \la \pi_w(g) \phi_w, \cg{\phi}_w \ra_{\pi_w} 
        = 
    \la \pi_w(Xg) \phi_w, \cg{\pi}_w(X) \cg{\phi}_w \ra_{\pi_w}
        = 
    \la \pi_w(Z_1) \phi_w, \cg{\pi}_w(Z_2) \cg{\phi}_w \ra_{\pi_w}.
\]

Therefore, using \eqref{def Schwartz Phi w}, we obtain
\begin{align*}
    Z_w &= 
    \frac{(\dim \tau_w)^2}{\vol(\GG_w)}
    \int_{\GL_n(\KK_w)} 
        \chi_{w,2}(Z_1) \chi_{w,1}(Z_2)^{-1} 
        \absv{ \det Z_1Z_2 }_w^{s+\frac{n}{2}} 
        \\ &\times 
        \Phi_{1,w}(Z_1', Z_2'') 
        \Phi_{2,w}(Z_2', Z_1'') 
        \la 
            \pi_w(Z_1) \phi_w, 
            \cg{\pi}_w(Z_2) \cg{\phi}_w
        \ra_{\pi_w} 
    d^\times Z_1 d^\times Z_2 \ .
\end{align*}

We take the integrals over the following open subsets of full measure. We take the integral in $Z_1$ over
\begin{align*}
    \left\{ 
	\begin{pmatrix}
		1 & 0 \\ 
            C_1 & 1
	\end{pmatrix}
	\begin{pmatrix}
		A_1 & 0 \\ 
            0 & D_1
	\end{pmatrix}
	\begin{pmatrix}
		1 & B_1 \\ 
            0 & 1
	\end{pmatrix} \mid
            B_1, \tp{C}_1 \in 
                M_{a_w \times b_w}(\KK_w), 
            A_1 \in 
                \GL_{a_w}(\KK_w), 
            D_1 \in 
                \GL_{b_w}(\KK_w) 
    \right\},
\end{align*}
with the measure
\[
    \absv{ 
        \det A_1^{b_w} 
        \det D_1^{-a_w}
    }_w 
    dC_1 
    d^\times A_1 
    d^\times D_1 
    dB_1 \ .
\]

Similarly, we take the integral in $Z_2$ over
\begin{align*}
    \left\{ 
	\begin{pmatrix}
		1 & B_2 \\ 0 & 1
	\end{pmatrix}
	\begin{pmatrix}
		A_2 & 0 \\ 0 & D_2
	\end{pmatrix}
	\begin{pmatrix}
		1 & 0 \\ C_2 & 1
	\end{pmatrix} \mid 
            B_2, \tp{C}_2 \in 
                M_{a_w \times b_w}(\KK_w), 
            A_2 \in 
                \GL_{a_w}(\KK_w), 
            D_2 \in 
                \GL_{b_w}(\KK_w) 
    \right\},
\end{align*}
with the measure
\[
    \absv{ 
        \det A_2^{b_w} 
        \det D_2^{-a_w} 
    }_w 
    dC_2 
    d^\times A_2 
    d^\times D_2 
    dB_2 \ .
\]

Therefore, one has
\begin{align*}
	\Phi_{1,w}(Z_1', Z_2'') &= 
	\Phi_{1,w}\left(
		\begin{pmatrix}
			A_1 & B_2D_2 \\
			C_1A_1 & D_2
		\end{pmatrix}
	\right)	
	\\
	\Phi_{2,w}(Z_2', Z_1'') &=
	\Phi_{2,w}\left(
		\begin{pmatrix}
			A_2 + B_2D_2C_2 & A_1B_1 \\
			D_2C_2 & C_1A_1B_1 + D_1
		\end{pmatrix}
	\right)	
\end{align*}
and both can be simplified by considering their support.
\begin{lemma} \label{Z_i conditions}
    The product 
    \[
        \Phi_{1,w}
        \left(
		\begin{pmatrix}
			A_1 & B_2D_2 \\
			C_1A_1 & D_2
		\end{pmatrix}
	\right)	
	\Phi_{2,w}
        \left(
		\begin{pmatrix}
			A_2 + B_2D_2C_2 & A_1B_1 \\
			D_2C_2 & C_1A_1B_1 + D_1
		\end{pmatrix}
	\right)	
    \]
    is zero unless all of the following conditions are met:
    \begin{align*}
        A_1 \in 
            I_{a_w, r}^0 \ \ \ ; \ \ \
        D_2 \in 
            I_{b_w, r}^0 \ \ \ &; \ \ \
        C_1 \in 
            \GG_l\ \ \ ; \ \ \
        B_2 \in 
            \GG_u \\
        B_1 \in 
            M_{a_w \times b_w}(\OO_w) \ \ \ &; \ \ \
        C_2 \in 
            M_{b_w \times a_w}(\OO_w) \\
        A_2 \in 
            p^{-r} M_{a_w \times a_w}(\OO_w) \ \ \ &; \ \ \
        D_1 \in 
            p^{-r} M_{b_w \times b_w}(\OO_w)
    \end{align*}
	
	Moreover, in this case, the product is equal to 
    $
        \mu_{a_w}(A_1) 
        \mu_{b_w}(D_2) 
        \Phi_{w}^{(1)}(A_2)
        \Phi_{w}^{(4)}(D_1)
    $.
\end{lemma}
\begin{proof}
	Using Lemma \ref{Phi2 support} and the definition of $\Phi_{w,1}$, it is clear that the product above is nonzero if and only if the conditions above are satisfied. Moreover, if they are satisfied, one has
    \[
        \Phi_{1,w}
        \left(
	    \begin{pmatrix}
                A_1 & B_2D_2 \\
                C_1A_1 & D_2
            \end{pmatrix}
        \right)	
	   = 
        \mu_{a_w}(A_1) 
        \mu_{b_w}(D_2) 
\]
    by definition of $\mu_w$. One also obtains
    \[	
	\Phi_{2,w}
        \left(
            \begin{pmatrix}
                A_2 + B_2D_2C_2 & A_1B_1 \\
                D_2C_2 & C_1A_1B_1 + D_1
            \end{pmatrix}
        \right)
        =
        \Phi_w^{(1)}(A_2)
        \Phi_w^{(4)}(D_1)
    \]
	as in the proof of Lemma \ref{Phi2 support}.
\end{proof}

\begin{lemma}
Under the conditions of Lemma \ref{Z_i conditions}, one has
\[
    \la 
        \pi_w(Z_1) \phi_w, 
        \cg{\pi}_w(Z_2) \cg{\phi}_w 
    \ra_{\pi_w}
    =
    \la 
        \pi_w
        \left(
            \begin{pmatrix}
                A_1 & 0 \\ 0 & 1
            \end{pmatrix}
        \right) 
        \phi_w,
        \cg{\pi}_w
        \left(
            \begin{pmatrix}
                A_2 & 0 \\ 0 & D_1^{-1}D_2
            \end{pmatrix}
        \right)
        \cg{\phi}_w
    \ra_{\pi_w}
\]
\end{lemma}

\begin{proof}
    We write
    \[
        Z_1 = \begin{pmatrix}
            1 & 0 \\ 
            C_1 & 1
        \end{pmatrix}
        \begin{pmatrix}
            A_1 & 0 \\ 
            0 & D_1
        \end{pmatrix}
        \begin{pmatrix}
            1 & B_1 \\ 
            0 & 1
        \end{pmatrix} 
        \ \ \ \text{and} \ \ \
        Z_2 = \begin{pmatrix}
            1 & B_2 \\ 
            0 & 1
        \end{pmatrix}
        \begin{pmatrix}
            A_2 & 0 \\ 
            0 & D_2
        \end{pmatrix}
        \begin{pmatrix}
            1 & 0 \\ 
            C_2 & 1
        \end{pmatrix}
    \]
    under the conditions of Lemma \ref{Z_i conditions}. As 
    $
        \begin{pmatrix}
            1 & B_1 \\
            0 & 1
        \end{pmatrix} 
        \in I_{w,r}
    $ 
    and
    $
        \begin{pmatrix}
            1 & 0 \\ 
            C_2 & 1
        \end{pmatrix} 
        \in \tp{I}_{w,r}
    $
    fix $\phi_w$ and $\cg{\phi}_w$ respectively, the pairing
    \[
        \la 
            \pi_w(Z_1)
            \phi_w, 
            \cg{\pi}_w(Z_2) 
            \cg{\phi}_w 
        \ra_{\pi_w}
    \]
    is equal to
    \begin{align*}
        \la
            \pi_w
            \left(
                \begin{pmatrix}
                    1 & -B_2 \\
                    0 & 1
                \end{pmatrix}
                \begin{pmatrix}
                    1 & 0 \\
                    C_1 & D_1
                \end{pmatrix}
            \right)
            \pi_w
            \left(
                \begin{pmatrix}
                    A_1	& 0 \\
                    0 & 1
                \end{pmatrix}
            \right)
            \phi_w,
            \cg{\pi}_w
           \left(
                \begin{pmatrix}
                    A_2 & 0 \\
                    0 & D_2
                \end{pmatrix}
            \right)
            \cg{\phi}_w
        \ra_{\pi_w}
    \end{align*}
	
    Furthermore, write
    \[
        \begin{pmatrix}
            1 & -B_2 \\
            0 & 1
        \end{pmatrix}
        \begin{pmatrix}
            1 & 0 \\
            C_1 & D_1
        \end{pmatrix} 
        =
        \begin{pmatrix}
            1 & 0 \\
            C & 1
        \end{pmatrix}
        \begin{pmatrix}
            A & 0 \\
            0 & D
        \end{pmatrix}
        \begin{pmatrix}
            1 & B \\
            0 & 1
        \end{pmatrix}
    \]
    where 
    \begin{align*}
        A &= 
            1 - B_2C_1 \in 
                1 + \p_w^{2r}M_{a_w}(\OO_w), \\
        CA &= 
            C_1 \in 
                \p_w^rM_{b_w \times a_w}(\OO_w), \\
        AB &= 
            -B_2D_1 \in 
                M_{a_w \times b_w}(\OO_w) \\
        D_1 &= 
            D + CAB \in 
                \p_w^{-r}M_{b_w}(\OO_w) \ .
    \end{align*}
    
    Note that $1 = A^{-1} + B_2C_1A^{-1}$, so
    \begin{align*}
        &A^{-1} = 
            1 - B_2C \in 
                1 + \p_w^{2r}M_{a_w}(\OO_w), \\
        &C \in 
            \p_w^rM_{b_w \times a_w}(\OO_w), \ \ \ \ \ 
        B \in 
            M_{a_w \times b_w}(\OO_w), \\
        &D = 
            (1+CB_2)D_1 \in 
                (1+\p_w^{2r})M_{b_w}(\OO_w)D_1 \ .
    \end{align*}
    
    Therefore,
    \begin{align*}
        \begin{pmatrix}
            1 & -B_2 \\
            0 & 1
        \end{pmatrix}
        \begin{pmatrix}
            1 & 0 \\
            C_1 & D_1
        \end{pmatrix} =
        \begin{pmatrix}
            1 & 0 \\
            C & 1 + CB_2
        \end{pmatrix}
        \begin{pmatrix}
            1 & 0 \\
            0 & D_1
        \end{pmatrix}
        \begin{pmatrix}
            A & AB \\
            0 & 1
        \end{pmatrix}
    \end{align*}
	
    Setting
    \begin{align*}	
        \gamma_0 = 
        \begin{pmatrix}
            A & AB \\
            0 & 1
        \end{pmatrix} 
        \ \ \ \text{ and } \ \ \
        \cg{\gamma}_0 =
        \begin{pmatrix}
            1 & 0 \\
            C & 1 + CB_2
        \end{pmatrix} \ ,
    \end{align*}
    one obtains
    \begin{align*}
        \la 
            &\pi_w(Z_1)\phi_w, 
            \cg{\pi}_w(Z_2) \cg{\phi}_w 
        \ra_{\pi_w} \\
            &= 
        \la 
            \pi_w
            \left(
                \gamma_0
                \begin{pmatrix}
                    A_1 & 0 \\ 
                    0 & 1
                \end{pmatrix}
            \right) 
            \phi_w,
            \cg{\pi}_w
            \left(
                \begin{pmatrix}
                    1 & 0 \\ 
                    0 & D_1^{-1}
                \end{pmatrix}
                \cg{\gamma}_0
                \begin{pmatrix}
                    A_2 & 0 \\ 
                    0 & D_2
                \end{pmatrix}
            \right) 
            \cg{\phi}_w
        \ra_{\pi_w}
    \end{align*}
	
    Then Lemma \ref{Iwahori fixed} yields the desired result since $\gamma_0, \tp{\cg{\gamma}_0} \in I_{w,r}$.
\end{proof}

\begin{proposition} \label{pairing pi phi}
    Under the conditions of Lemma \ref{Z_i conditions}, we have 
    \[
        \Phi_{w,1}(Z_1', Z_2'') 
        \Phi_{w,2}(Z_2', Z_1'') 
        \la 
            \pi_w(Z_1) 
            \phi_w, 
            \cg{\pi}_w(Z_2) 
            \cg{\phi}_w
        \ra_{\pi_w} =
        \vol(I^0_{a_w,b_w, r}) 
        \cdot J_{a_w} 
        \cdot J_{b_w}
    \]
    where
    \begin{align*}
        J_{a_w} 
        &= 
            \mu_{a_w}(A_1) 
            \Phi_w^{(1)}(A_2)
            \absv{
                \det A_2
            }_w^{b_w/2}
            \la 
                \pi_{a_w}(A_1) 
                \phi_{a_w}, 
                \cg{\pi}_{a_w}(A_2) 
                \cg{\phi}_{a_w}
            \ra_{\pi_{a_w}}, 
        \\
        J_{b_w} 
        &= 
            \mu_{b_w}(D_2) 
            \Phi_w^{(4)}(D_1)
            \absv{
                \det D_1
            }_w^{a_w/2}
            \la 
                \pi_{b_w}(D_1) 
                \phi_{b_w},
                \cg{\pi}_{b_w}(D_2)
                \cg{\phi}_{b_w}
            \ra_{\pi_{b_w}}
    \end{align*}
\end{proposition}

\begin{proof}
    Using the conditions on $Z_1$ and $Z_2$, we have
    \begin{align*}	
        &\la 
            \pi_w(Z_1)
            \phi_w, 
            \cg{\pi}_w(Z_2) 
            \cg{\phi}_w 
        \ra_{\pi_w} \\
        &= 
        \la 
            \pi_w
            \left(
                \begin{pmatrix}
                    A_1 & 0 \\ 
                    0 & 1
                \end{pmatrix}
            \right) 
            \phi_w,
            \cg{\pi}_w
            \left(
                \begin{pmatrix}
                    A_2 & 0 \\ 
                    0 & D_1^{-1}D_2
                \end{pmatrix}
            \right) 
            \cg{\phi}_w
        \ra_{\pi_w} \\
        &= 
        \int_{\GL_n(\OO_w)} 
            (
                \varphi_w
                \left(
                    k
                    \begin{pmatrix}
                        A_1 & 0 \\ 
                        0 & 1
                    \end{pmatrix}
                \right),
                \tilde{\varphi}_w
                \left(
                    k
                    \begin{pmatrix}
                        A_2 & 0 \\ 
                        0 & D_1^{-1}D_2
                   \end{pmatrix}
                \right)
            )_{w} 
        d^\times k \ ,
    \end{align*}
    using Equation \eqref{inner product pi w integral}.
	
    As the support of $\varphi_w$ is $P_{a_w,b_w} I_{w,r}$ and its intersection with $\GL_n(\OO_w)$ is equal to $I^0_{a_w,b_w,r}$, the integrand above is nonzero if and only if $k \in I^0_{a_w,b_w,r}$. Write such a $k \in I^0_{a_w,b_w,r}$ as
    \[
        k = 
        \begin{pmatrix}
            1 & B \\ 0 & 1
        \end{pmatrix}
        \begin{pmatrix}
            A & 0 \\ 0 & D
        \end{pmatrix}
        \begin{pmatrix}
            1 & 0 \\ C & 1
        \end{pmatrix}
    \]
    with 
    $
        A \in \GL_{a_w}(\OO_w)
    $, 
    $
        D \in \GL_{b_w}(\OO_w)
    $, 
    $
        B \in M_{a_w \times b_w}(\OO_w)
    $ and 
    $
        C \in \p_w^r M_{b_w \times a_w}(\OO_w)
    $.
    Then, using Lemma \ref{Iwahori fixed}, one obtains
    \begin{align*}
        \varphi_w
        \left(
            k
            \begin{pmatrix}
                A_1 & 0 \\ 
                0 & 1
            \end{pmatrix}
        \right)
        &= 
        \varphi_w \left( 
        \begin{pmatrix}
            AA_1 & 0 \\ 0 & D
        \end{pmatrix}
        \right) \\
        \cg{\varphi}_w
        \left(
            k
            \begin{pmatrix}
                A_2 & 0 \\ 
                0 & D_1^{-1}D_2
            \end{pmatrix}
        \right) 
        &=
        \cg{\varphi}_w
        \left(
            \begin{pmatrix}
                AA_2 & 0 \\ 
                0 & DD_1^{-1}D_2
            \end{pmatrix}
        \right) 
    \end{align*}
    
    Observe that the determinant of the matrices $A$, $D$, $A_1$ and $D_2$ are all integral $p$-adic units. Therefore, using the definition of $\varphi_w$ (resp. $\cg{\varphi}_w$) and its relation to $\phi_{a_w} \otimes \phi_{b_w}$ (resp. $\cg{\phi}_{a_w} \otimes \cg{\phi}_{b_w})$, the integrand above is equal to
    \begin{align*}
        &
        \absv{
            \det A_2
        }_w^{b_w/2} 
        \absv{
            \det D_1^{-1}
        }_w^{-a_w/2} \\
        &\times
        \la 
            \pi_{a_w}(AA_1) \phi_{a_w}
            \otimes 
            \pi_{b_w}(D) \phi_{b_w},
            \cg{\pi}_{a_w}(AA_2) \cg{\phi}_{a_w} 
            \otimes 
            \cg{\pi}_{b_w}(DD_1^{-1}D_2) \cg{\phi}_{b_w}
        \ra_{a_w,b_w} \\
        &=
        \absv{
            \det A_2
        }_w^{b_w/2} 
        \absv{
            \det D_1
        }_w^{a_w/2}
        \la 
            \pi_{a_w}(A_1) \phi_{a_w}, 
            \cg{\pi}_{a_w}(A_2) \cg{\phi}_{a_w}
        \ra_{\pi_{a_w}}
        \la 
            \pi_{b_w}(D_1) \phi_{b_w},
            \cg{\pi}_{b_w}(D_2) \cg{\phi}_{b_w}
        \ra_{\pi_{b_w}} \ ,
    \end{align*}
    which does not depend on $k$. The result follows by using the second part of Lemma \ref{Z_i conditions}.
\end{proof}

\begin{corollary} \label{corZwIaIb}
    The zeta integral $Z_w$ is equal to
    \[
        \frac{
            (\dim \tau_w)^2 \vol(I^0_{a_w, b_w, r})
        }{
            \vol(I_{a_w, r}^0)\vol(I_{b_w, r}^0)
        } \cdot 
        \mathcal{I}_{a_w} \cdot \mathcal{I}_{b_w}
    \]
    where
    \begin{align*}
        \mathcal{I}_{a_w} = 
        \int_{I_{a_w, r}^0} 
            &\int_{\GL_{a_w}(\KK_w)} 
                \mu_{a_w}(A_1) 
                \chi_{w,2}(A_1) 
                \chi_{w,1}^{-1}(A_2) 
                \\
                &\times
                \Phi_w^{(1)}(A_2) 
                \absv{
                    \det A_2
                }_w^{s + \frac{a_w}{2}} 
                \la 
                    \pi_{a_w}(A_1) 
                    \phi_{a_w}, 
                    \cg{\pi}_{a_w}(A_2)
                    \cg{\phi}_{a_w} 
                \ra_{\pi_{a_w}} 
            d^\times A_2 
        d^\times A_1 \\
        \mathcal{I}_{b_w} = 
        \int_{I_{b_w, r}^0} 
            &\int_{\GL_{b_w}(\KK_w)} 
                \mu_{b_w}(D_2) 
                \chi_{w,2}(D_1) 
                \chi_{w,1}^{-1}(D_2)
                \\
                &\times
                \Phi_w^{(4)}(D_1) 
                \absv{
                    \det D_1
                }_w^{s + \frac{b_w}{2}} 
                \la 
                    \pi_{b_w}(D_1) 
                    \phi_{b_w}, 
                    \cg{\pi}_{b_w}(D_2) 
                    \cg{\phi}_{b_w} 
                \ra_{\pi_{b_w}} 
            d^\times D_1 
        d^\times D_2
    \end{align*}
\end{corollary}

\begin{proof}
    Using Lemma \ref{Z_i conditions} and Lemma \ref{pairing pi phi},
    \begin{align*}
        Z_w 
        &= 
            \frac{(\dim \tau_w)^2}{\vol(\GG_w)} \\
            &\times 
            \int_{
                A_1, B_1, C_1, A_2, B_2, D_1, C_2, D_2
            } 
                \chi_{w,2}(A_1)
                \chi_{w,2}(D_1) 
                \chi_{w,1}^{-1}(A_2)
                \chi_{w,1}^{-1}(D_2) \\
                &\times 
                \absv{
                    \det A_1 
                    \det D_1 
                    \det A_2 
                    \det D_2
                }_w^{s+\frac{n}{2}} \\
                &\times 
                \vol(
                    I^0_{a_w, b_w, r}
                )
                J_{a_w} J_{b_w} \\
            &\times 
            \absv{
                \det A_1^{b_w} 
                \det D_1^{-a_w}
            } 
            d^\times A_1 
            dB_1 
            dC_1 
            d^\times D_1 \\
            &\times 
            \absv{
                \det A_2^{-b_w}
                \det D_2^{a_w}
            } 
            d^\times A_2 
            dB_2 
            dC_2 
            d^\times D_2
    \end{align*}
    where the domain of integration for the matrices $A_i$, $B_i$, $C_i$ and $D_i$ ($i=1,2$) is given by the conditions of Lemma \ref{Z_i conditions}.
	
    Note that the integrand is independent of $B_1 \in M_{a_w \times b_w}(\OO_w)$, $B_2 \in \GG^u$, $C_1 \in \GG^l$ and $C_2 \in M_{b_w \times a_w}(\OO_w)$. Moreover, the determinants of the matrices $A_1$ and $D_2$ are both $p$-adic units. Therefore, the above simplifies to
    \begin{align*}	
        Z_w 
        &= 
            \frac{(\dim \tau_w)^2}{\vol(\GG_w)} 
            \vol(I^0_{a_w, b_w, r})
            \vol(M_{a_w \times b_w}(\OO_w))^2 \vol(\GG^l) 
            \vol(\GG^u) \\
            &\times 
            \int_{I_{a_w, r}^0} 
            \int_{I_{b_w, r}^0} 
            \int_{\GL_{a_w}(\KK_w)} 
            \int_{\GL_{b_w}(\KK_w)}
                \chi_{w,2}(A_1)
                \chi_{w,2}(D_1) 
                \chi_{w,1}^{-1}(A_2)
                \chi_{w,1}^{-1}(D_2) \\
                &\times 
                \mu_{a_w}(A_1) 
                \Phi_w^{(1)}(A_2)
                \absv{
                    \det A_2
                }_w^{s+\frac{a_w}{2}}
                \la 
                    \pi_{a_w}(A_1) 
                    \phi_{a_w}, 
                    \cg{\pi}_{a_w}(A_2) 
                    \cg{\phi}_{a_w}
                \ra_{\pi_{a_w}} \\
                &\times
                \mu_{b_w}(D_2) 
                \Phi_w^{(4)}(D_1)
                \absv{
                    \det D_1
                }_w^{s+\frac{b_w}{2}}
                \la 
                    \pi_{b_w}(D_1)
                    \phi_{b_w},
                    \cg{\pi}_{b_w}(D_2) 
                    \cg{\phi}_{b_w}
                \ra_{\pi_{b_w}} \\
            &\times 
            d^\times D_1 
            d^\times A_2 
            d^\times D_2 
            d^\times A_1 \ ,
    \end{align*}	
    and using the decomposition $\GG_w = \GG^l (I_{a_w, r}^0 \times I_{b_w, r}^0) \GG^u$, the result follows immediately.
\end{proof}

\begin{theorem} \label{main thm - comp of Zw}
    The integrals $\mathcal{I}_{a_w}$ and $\mathcal{I}_{b_w}$ are equal to
    \begin{align*}
        \mathcal{I}_{a_w} &= 
        \displaystyle{
            \frac{
                \epsilon(
                    -s + \frac{1}{2}, 
                    \pi_{a_w} 
                    \otimes 
                    \chi_{w,1}
                ) 
                L(
                    s + \frac{1}{2}, 
                    \cg{\pi}_{a_w} 
                    \otimes 
                    \chi_{w,1}^{-1}
                )
            }{
                L(
                    -s+\frac{1}{2},
                    \pi_{a_w} 
                    \otimes 
                    \chi_{w,1}
                )
            } \cdot 
            \frac{
                \vol(\X^{(1)}) \vol(I_{a_w, r}^0)
            }
            {
                (\dim \tau_{a_w})^2
            } 
            \la 
                \phi_{a_w}, 
                \cg{\phi}_{a_w} 
            \ra_{\pi_{a_w}}
        } \\
        \mathcal{I}_{b_w} &= 
        \displaystyle{
            \frac{
                L(
                    s+\frac{1}{2},
                    \pi_{b_w} 
                    \otimes 
                    \chi_{w,2})
            }{
                \epsilon(
                    s + \frac{1}{2},
                    \pi_{b_w} 
                    \otimes 
                    \chi_{w,2}) 
                L(
                    -s + \frac{1}{2},
                    \cg{\pi}_{b_w} 
                    \otimes 
                    \chi_{w,2}^{-1})
            } \cdot 			
            \frac{
                \vol(\X^{(4)}) \vol(I_{b_w, r}^0)
            }
            {
                (\dim \tau_{b_w})^2
            } \la 
                \phi_{b_w}, 
                \cg{\phi}_{b_w} 
            \ra_{\pi_{b_w}}
        }
    \end{align*}
    Therefore, by setting 
    \[
        L\left(
            s + \frac{1}{2},
            \ord,
            \pi_w, 
            \chi_w
        \right) := 
        \displaystyle{
            \frac{
                \epsilon(
                    -s + \frac{1}{2},
                    \pi_{a_w}
                    \otimes 
                    \chi_{w,1}) 
                L(
                    s + \frac{1}{2},
                    \cg{\pi}_{a_w}
                    \otimes 
                    \chi_{w,1}^{-1})
                L(
                    s+\frac{1}{2}, 
                    \pi_{b_w} 
                    \otimes 
                    \chi_{w,2})
            }{
                L(
                    -s+\frac{1}{2}, 
                    \pi_{a_w} 
                    \otimes 
                    \chi_{w,1})
                \epsilon(
                    s + \frac{1}{2},
                    \pi_{b_w} 
                    \otimes 
                    \chi_{w,2}
                ) 
                L(
                    -s + \frac{1}{2},
                    \cg{\pi}_{b_w} 
                    \otimes 
                    \chi_{w,2}^{-1}
                )
            }
        }
    \]
    one has
    \[
        Z_w = 
        L\left(
            s + \frac{1}{2},
            \ord, 
            \pi_w, 
            \chi_w
        \right) 
            \cdot 
        \frac{
            \vol(I_{w, r}^0)
            \vol(\tp{I}_{w, r}^0)
        }{
            \vol(I_{w, r}^0 \cap 
            \tp{I}_{w, r}^0)
        }
            \cdot 
        \la 
            \varphi_w, 
            \cg{\varphi}_w
        \ra_{\pi_w}
    \]
\end{theorem}

\begin{proof}
    This proof is inspired by the argument of \cite[Theorem 4.3.10]{EHLS}. First, write
    \begin{align*}
        \mathcal{I}_{a_w} = 
        \int_{I_{a_w, r}^0}
            \mu_{a_w}(A_1) 
            \chi_{w,2}(A_1)
             \mathcal{I}_{a_w, 2}(A_1)
        d^\times A_1 \ ,
    \end{align*}
    where $\mathcal{I}_{a_w, 2} = \mathcal{I}_{a_w, 2}(A_1)$ is defined as
    \begin{align*}
        \int_{\GL_{a_w}(\KK_w)} 
                \Phi_w^{(1)}(A_2) 
                \absv{
                    \det A_2
                }_w^{s + \frac{a_w}{2}}
                \la
                    \pi_{a_w}(A_1)
                    \phi_{a_w}, 
                    (
                        \chi_{w,1}^{-1} 
                        \otimes 
                        \cg{\pi}_{a_w}
                    )(A_2)
                    \cg{\phi}_{a_w}
                \ra_{\pi_{a_w}} 
            d^\times A_2 \ .
    \end{align*}
    
    The above is a ``Godement-Jacquet'' integral, as defined in \cite[Equation (1.1.3)]{Jac79}. Therefore, we use its functional equation to obtain
    \begin{align*}
        &\mathcal{I}_{a_w,2} = 
            \frac{
                \epsilon(
                    -s + \frac{1}{2},
                    \pi_{a_w} 
                    \otimes 
                    \chi_{w,1}
                ) 
                L(
                    s + \frac{1}{2},
                    \cg{\pi}_{a_w} 
                    \otimes 
                    \chi_{w,1}^{-1}
                )
            }{
                L(
                    -s+\frac{1}{2},
                    \pi_{a_w} 
                    \otimes 
                    \chi_{w,1}
                )
            } \\
            &\times
            \int_{\GL_{a_w}(\KK_w)}
                \left( 
                    \Phi_w^{(1)} 
                \right)^{\wedge}(A_2)
                \absv{
                    \det A_2
                }_w^{-s + \frac{a_w}{2}}
                \chi_{w,1}(A_2) 
                \la 
                    \pi_{a_w}(A_1) 
                    \phi_{a_w},
                    \cg{\pi}_{a_w}(A_2^{-1})
                    \cg{\phi}_{a_w}
                \ra_{\pi_{a_w}}
            d^\times A_2
    \end{align*}
    
    Let $L_{a_w, \ord}$ denote the quotient of $L$-factors and $\epsilon$-factors leading the expression above. Recall that 
    $
        \left( 
            \Phi_w^{(1)} 
        \right)^{\wedge}(A_2)
    $ is supported on $\X^{(1)}$. Furthermore, for $A_2 \in \X^{(1)}$, we have 
    $
        \left( 
            \Phi_w^{(1)} 
        \right)^{\wedge}(A_2)
            =
        \nu_{a_w}(A_2)
    $
    and 
    $
        \absv{
            \det A_2
        }_w = 1
    $. Then, $\mathcal{I}_{a_w,2}$ is equal to
    \begin{align*}
        L_{a_w, \ord} \times	
        \int_{\X^{(1)}}
            \chi_{w,1}(A_2)
            \nu_{a_w}(A_2)
            \la 
                \pi_{a_w}(A_1) 
                \phi_{a_w},
                \cg{\pi}_{a_w}(A_2^{-1}) 
                \cg{\phi}_{a_w}
            \ra_{\pi_{a_w}}
        d^\times A_2
    \end{align*}

    By definition of $\X^{(1)}$, we can write $A_2 = \gamma_1k_2\gamma_2$ uniquely for some $k_2 \in K_{a_w}$, $\gamma_1 \in \X^{(1)}_l := \tp{I}^0_{a_w, r} \cap \tp{P}_{a_w}^u$, and $\gamma_2 \in \X^{(1)}_u := I_{a_w, r}^0 \cap P_{a_w}^u$. It follows that
    \begin{align*}
        \la 
            \pi_{a_w}(A_1) 
            \phi_{a_w},
            \cg{\pi}_{a_w}(A_2^{-1}) 
            \cg{\phi}_{a_w}
        \ra_{\pi_{a_w}} 
        &=
        \la 
            \pi_{a_w}(k_2 \gamma_2 A_1) 
            \phi_{a_w},
            \cg{\pi}_{a_w}(\gamma_1^{-1}) 
            \cg{\phi}_{a_w}
        \ra_{\pi_{a_w}} \\
        &=
        \la 
            \pi_{a_w}(k_2 A_1) 
            \phi_{a_w},
            \cg{\phi}_{a_w}
        \ra_{\pi_{a_w}} \\
        &=
        \int_{\GL_{a_w}(\OO_w)}
            (
                \varphi_{a_w} (k k_2 A_1),
                \cg{\varphi}_{a_w}(k)
            )_{a_w}
        d^\times k
    \end{align*}
    
    The support of $\varphi_{a_w}$ is $P_{a_w} I_{a_w, r} = P_{a_w} I^0_{a_w, r}$. Since $k_2 A_1 \in I^0_{a_w, r}$, the integrand vanishes unless $k \in P_{a_w} I_{a_w, r} \cap \GL_{a_w}(\OO_w) = I^0_{a_w, r}$. Using the fact that such $k$ is in $P_{a_w}$ as well as Equation \eqref{def varphi aw}, we obtain
    \[
        (
            \varphi_{a_w} (k k_2 A_1),
            \cg{\varphi}_a(k)
        )_{a_w}
        =
        (
            \varphi_{a_w} (k_2 A_1),
            \cg{\varphi}_{a_w}(1)
        )_{a_w}
        =
        (
            \tau_{a_w}(k_2 A_1) 
            \phi^0_{a_w},
            \cg{\phi}^0_{a_w}
        )_{a_w} .
    \]
    
    Then, using the above, Equation \eqref{def extn mu aw}, the definition of $\nu_{a_w}$, and orthogonality relations of matrix coefficients, we obtain
    \begin{align*}
        \mathcal{I}_{a_w, 2}
            &=
        L_{a_w, \ord} \times
        \int_{\X^{(1)}}
            \mu'_{a_w}(A_2)
            \la 
                \pi_{a_w}(A_1) 
                \phi_{a_w},
                \cg{\pi}_{a_w}(A_2^{-1}) 
                \cg{\phi}_{a_w}
            \ra_{\pi_{a_w}}
        d^\times A_2 \\
            &=
        L_{a_w, \ord}
        \vol(I_{a_w, r}^0)
        \vol(\X^{(1)}_l) 
        \vol(\X^{(1)}_u) \\
            &\times
        \int_{K_{a_w}}
            ( 
                \phi^0_{a_w},
                \cg{\tau}_{a_w}(k_2)
                \cg{\phi}^0_{a_w}
            )_{a_w}
            (
                \tau_{a_w}(k_2)
                \tau_{a_w}(A_1) 
                \phi^0_{a_w},
                \cg{\phi}^0_{a_w}
            )_{a_w}
        d^\times k_2 \\
            &=
        L_{a_w, \ord}
        \vol(I_{a_w, r}^0)
        \frac{
            \vol(\X^{(1)})
        }{
            \dim \tau_{a_w}
        }
        ( 
            \phi^0_{a_w},
            \cg{\phi}^0_{a_w}
        )_{a_w}
        (
            \tau_{a_w}(A_1) 
            \phi^0_{a_w},
            \cg{\phi}^0_{a_w}
        )_{a_w}
    \end{align*}
    
    Using Equation \eqref{size phia}, orthogonality relations of matrix coefficients once more, and the normalization $(\phi^0_{a_w}, \cg{\phi}^0_{a_w})_{a_w} = 1$, we ultimately obtain that $\mathcal{I}_{a_w}$ is equal to
    \begin{align*}
        &L_{a_w, \ord}
        \la
            \phi_{a_w}, 
            \cg{\phi}_{a_w} 
        \ra_{\pi_{a_w}}
        \frac{
            \vol(\X^{(1)})
        }
        {
            \dim \tau_{a_w}
        } 
        \int_{I_{a_w, r}^0} 
            \mu'_{a_w}(A_1)
            (
                \tau_{a_w}(A_1) 
                \phi^0_{a_w},
                \cg{\phi}^0_{a_w}
            )_{a_w}
        d^\times A_1 \\
        = 
        &L_{a, \ord}
        \frac{
            \vol(\X^{(1)}) 
            \vol(I_{a_w, r}^0)
        }
        {
            (\dim \tau_{a_w})^2
        } 
        \la 
            \phi_{a_w}, 
            \cg{\phi}_{a_w} 
        \ra_{\pi_{a_w}}
    \end{align*}
    
    A similar argument yields
    \begin{align*}
        \mathcal{I}_{b_w}
        &=
        L_{b_w, \ord}
        \frac{
            \vol(\X^{(4)})
            \vol(I_{b_w, r}^0)
        }
        {
            (\dim \tau_{b_w})^2
        }
        \la 
            \phi_{b_w}, 
            \cg{\phi}_{b_w} 
        \ra_{\pi_{b_w}} \ ,
    \end{align*}
    where
    \begin{align*}
        L_{b_w, \ord} =
        \frac{
            L(
                s+\frac{1}{2},
                \pi_{b_w} 
                \otimes 
                \chi_{w,2}
            )
        }{
            \epsilon(
                s + \frac{1}{2},
                \pi_{b_w} 
                \otimes 
                \chi_{w,2}
            ) 
            L(
                -s + \frac{1}{2},
                \cg{\pi}_{b_w} 
                \otimes 
                \chi_{w,2}^{-1}
            ) \ .
        }
    \end{align*}
    
    Therefore, the result follows using Corollary \ref{corZwIaIb}, Equation \eqref{size varphi}, and the identity
    \[
        \vol(\X^{(1)})\vol(\X^{(4)}) 
        = 
            \frac{
                \vol(I_{a_w, r}^0)
                \vol(\tp{I}_{a_w, r}^0)
                \vol(I_{b_w, r}^0)
                \vol(\tp{I}_{b_w, r}^0)
            }{
                \vol(I_{a_w, r}^0 \cap 
                \tp{I}_{a_w, r}^0)
                \vol(I_{b_w, r}^0 \cap 
                \tp{I}_{b_w, r}^0)
            } 
        =
            \frac{
                \vol(I_{w, r}^0)
                \vol(\tp{I}_{w, r}^0)
            }{
                \vol(I_{w, r}^0 \cap 
                \tp{I}_{w, r}^0)
            }    
    \]
\end{proof}

\subsection{Main Local Theorem} \label{section main local thm}
Keeping with the notation of Theorem \ref{main thm - comp of Zw}, define
\[
    E_p\left( 
        s+\frac{1}{2}, \Pord, \pi_p, \chi_p
    \right)
        := 
    \prod_{w \in \Sigma_p} 
        L\left( 
            s+\frac{1}{2}, \Pord, \pi_w, \chi_w
        \right).
\]

Then, from Theorem \ref{main thm - comp of Zw} and \eqref{def local zeta integrals}, we immediately obtain our main result.

\begin{theorem} \label{main local thm}
    Let $\chi$ be a unitary Hecke character of $\KK$, $\chi_p = \otimes_{w \mid p} \chi_w$, and let $s \in \CC$. Let $f_p^+ \in I_p(\chi_p, s)$ be the local Siegel-Weil section defined in Section \ref{construction f+} and let $\varphi_p \in \pi_p$, $\cg{\varphi}_p \in \cg{\pi}_p$ be the test vectors defined in \eqref{def test vec varphi p and cg varphi p}.
    
    Then, the $p$-adic local zeta integral $I_p(\varphi_p, \cg{\varphi}_p, f_p^+; \chi_p, s)$ is equal to
    \[
        E_p\left( 
            s+\frac{1}{2}, \Pord, \pi_p, \chi_p
        \right)
            \cdot
        \frac{
            \vol(I_{P, r}^0)
            \vol(\tp{I}_{P, r}^0)
        }{
            \vol(I_{P, r}^0 \cap 
            \tp{I}_{P, r}^0)
        }
    \]
\end{theorem}

\begin{remark}
    Using the same minor manipulation explained in \cite[Remark 4.3.11]{EHLS}, we see that the $p$-Euler factor $E_p(s+\frac{1}{2}, \Pord, \pi_p, \chi_p)$ takes the form of a modified Euler factor at $p$ as predicted in \cite[Section 2, Equation (18b)]{Coa89} for the conjectures of Coates and Perrin-Riou on $p$-adic $L$-functions.
\end{remark}